\documentclass[a4paper,11pt]{article}
\usepackage{geometry}
\usepackage[utf8]{inputenc} % allow utf-8 input
\usepackage[T1]{fontenc}    % use 8-bit T1 fonts
\usepackage{url}            % simple URL typesetting
\usepackage{booktabs}       % professional-quality tables
\usepackage{nicefrac}       % compact symbols for 1/2, etc.
\usepackage{microtype}      % microtypography

\usepackage{times}
\usepackage{graphicx} % more modern
\usepackage{subfigure} 
\usepackage{amsfonts,amsmath,amssymb}
\usepackage{color}
\usepackage{amsthm}
\usepackage{comment}
\usepackage{paralist}
\usepackage{enumitem}
\setitemize{noitemsep,topsep=0pt,parsep=0pt,partopsep=0pt}

\usepackage{mathtools}
\DeclarePairedDelimiter{\ceil}{\lceil}{\rceil}
 
\usepackage{algorithm}
\usepackage{algorithmic}
\usepackage[normalem]{ulem}

\usepackage{caption}
\usepackage[sort,nocompress]{cite}

\usepackage{empheq}

\usepackage{url}
\usepackage[colorlinks = true, pdfstartview = FitV, linkcolor = blue, citecolor = blue, urlcolor = blue]{hyperref}

\usepackage{lipsum}

\newcommand\blfootnote[1]{%
  \begingroup
  \renewcommand\thefootnote{}\footnote{#1}%
  \addtocounter{footnote}{-1}%
  \endgroup
}

\newtheorem{lemma}{Lemma}

\newcommand {\bss}  { {\bf s} }

\newcommand {\rr}  { {\bf r} }

\newcommand {\qq}  { {\bf q} }

\newcommand {\pp}  { {\bf p} }

\newcommand {\dd}  { {\bf d} }
\newcommand {\ee}  { {\bf e} }

\newcommand*\lin[1]{\langle #1\rangle}

\newtheorem{theorem}{Theorem}
\newtheorem{remark}{Remark}
\newcommand{\G}{\mathcal{G}}
\newcommand{\V}{\mathcal{V}}

\begin{document}

\title{Variational Perspective on Local Graph Clustering\blfootnote{A preliminary version of this work appeared with the title ``Exploiting Optimization for Local Graph Clustering'' as a technical report \cite{FCSRM16_TR}.}}

%\subtitle{Do you have a subtitle?\\ If so, write it here}

%\titlerunning{Short form of title}        % if too long for running head

\author{
        Kimon~Fountoulakis \and
        Farbod~Roosta-Khorasani \and
        Julian~Shun\and
        Xiang~Cheng \and
        Michael~W.~Mahoney
}

%\authorrunning{Short form of author list} % if too long for running head

%\institute{Kimon~Fountoulakis \at
%					Department of Statistics, UC Berkeley, Berkeley, CA 94720\\
%					\email{kfount@berkeley.edu}           %  \\
%					\and
%					Farbod~Roosta-Khorasani \at
%					Department of Statistics, UC Berkeley, Berkeley, CA 94720\\
%					\email{farbod@stat.berkeley.edu}   
%					 \and
%					Julian~Shun \at
%					Department of Electrical Engineering and Computer Science, UC Berkeley, Berkeley, CA, 94720\\
%					\email{jshun@eecs.berkeley.edu}   
%					\and
%          	 			Xiang~Cheng \at
%					Department of Electrical Engineering and Computer Science, UC Berkeley, Berkeley, CA, 94720\\
%					\email{x.cheng@berkeley.edu}
%					\and
%					Michael~W.~Mahoney \at
%					Department of Statistics, UC Berkeley, Berkeley, CA 94720\\
%					\email{mmahoney@stat.berkeley.edu}   
%}

\date{December 4, 2017}
% The correct dates will be entered by the editor

\maketitle

\begin{abstract}
Modern graph clustering applications require the analysis of large graphs and this can be computationally expensive. In this regard, local spectral graph clustering methods aim to identify well-connected clusters around a given ``seed set'' of 
reference nodes without accessing the entire graph. The celebrated Approximate Personalized PageRank (APPR) algorithm in the seminal paper by Andersen et al.\ \cite{ACL06} is one such method. 
APPR was introduced and motivated purely from an algorithmic perspective. In other words, there is no a priori notion of objective function/optimality conditions that characterizes the steps taken by APPR. Here, we derive a novel variational formulation which makes explicit the actual optimization problem solved by APPR. In doing so, we draw connections between the local spectral algorithm of~\cite{ACL06} and an iterative shrinkage-thresholding algorithm (ISTA). 
In particular, we show that, appropriately initialized ISTA applied to our variational formulation can recover the sought-after local cluster in a time that only depends on the number of non-zeros of the optimal solution instead of the entire graph. In the process, we show that an optimization algorithm which apparently requires accessing the entire graph, can be made to behave in a completely local manner by accessing only a small number of nodes. This viewpoint builds a bridge across two seemingly disjoint fields of graph processing and numerical optimization, and it allows one to leverage well-studied, numerically robust, and efficient optimization algorithms for processing today’s large graphs. 

%\keywords{Local Spectral Graph Clustering \and Variational Formulation \and Approximate Personalized PageRank \and Iterative Shrinkage-Thresholding}
% \PACS{PACS code1 \and PACS code2 \and more}
%\subclass{05C85    \and 90C35     \and 65K10   }

\end{abstract}

\section{Introduction}
\label{sec:intro}
Modern graph clustering applications require the analysis of large graphs \cite{JBPMM15,LLDM11}. However, in many cases, large sizes of recent graph data have rendered the applications of classical ``global'' approaches, i.e., those that require access to the entire graph, e.g., \cite{ARV09,GS06,H70,LR88,PSL90}, rather impractical. The requirement to access the entire graph is indeed very undesirable. This is so since, the running time of these global algorithms typically increases with the size of the entire graph. 
This computational challenge sparked the development of more recent methods \cite{ACL06,AL08,ST13,OZ14,KG14, veldticml2016} that are \emph{local} and only require access to a small portion of the graph. More specifically, given a ``target'' cluster, such local methods find a ``nearby'' cluster that sufficiently overlaps with the target and also has certain similar mathematical properties.
Unlike global methods, the running time of these local alternatives depends only on the size of the output cluster or on the size of an input seed set of reference nodes, both of which can be significantly smaller than the entire graph.\
This property makes local graph clustering methods more applicable for today's large-scale graphs. In addition, many real-world graphs tend to have ``good'' small/medium size local clusters, as opposed to ``good'' large ones \cite{LLDM11,JBPMM15}, making the application of such local algorithms even more appealing in practice.\footnote{In between global and local algorithms, there is a class of \emph{locally-biased algorithms}, e.g.,~\cite{MOV12}, whose running time depends on the entire graph, however, the solution is locally-biased toward some input seed set of reference nodes. We don't consider them in this paper.}

Approximate Personalized PageRank (APPR) algorithm, first introduced in the seminal paper~\cite{ACL06}, has been the cornerstone of local spectral graph clustering algorithms.
APPR is a semi-supervised approximation algorithm for finding local partitions in a graph, and it does so by approximately solving the PageRank linear system, followed by rounding the approximate solution (see Section \ref{sec:background} for more details). Heuristic modifications of APPR have also been proposed which have successfully aimed at improving its performance, e.g., those that use different rules to update the iterates and/or to terminate iterations~\cite{GM14_ICML}.
%Heuristic modifications of APPR such as those that use different termination criteria \fred{The following statement about probability mass is too detailed for introduction and I suggest removing it from here} or different rules for pushing probability mass to neighbors are applied in practice in order to achieve improved performance (\textcolor{black}{xxx, @Michael, any citations? I got this form the anti-differentiation paper}). 
However, APPR was introduced and motivated purely from an algorithmic perspective. As a result, its output is solely determined by the operations of the algorithm applied to the data. In other words, there is no a priori notion of objective function/optimality conditions that characterizes the steps taken by APPR. 
As a result, it is often difficult to precisely quantify how such heuristic modifications affect the theoretical guarantees and the running time of APPR. \emph{Our main objective here is to bridge this gap between APPR's theory and its heuristic modifications}. We do this by finding the \emph{explicit variational formulation} of the local graph clustering problem, which is only implicitly considered in APPR. This viewpoint indeed decouples the combinatorial properties of the graph from the characteristics of the optimization algorithm used to solve the new formulation. More importantly, we will demonstrate that by using a popular optimization algorithm, namely iterative shrinkage-thresholding algorithm (ISTA),~\cite{sra2012optimization}, and with proper initialization, one can indeed  guarantee similar local properties as those of APPR. The ``big-picture'' objective of this work is to build a bridge between two seemingly disjoint fields of graph processing and numerical optimization. It is hoped that once this viewpoint is extended to other graph processing problems, faster and more efficient algorithms emerge as a result.

In light of the aforementioned goals, our contributions can be summarized as follows. In comparison to APPR in which the properties of the local/sparse solutions and those of the employed algorithms are tightly coupled, we propose a variational formulation in the form of $\ell_1$-regularized PageRank (PR) that decouples the locality/sparsity of the solution from properties of the algorithm. In other words, if there exists a local solution for the original clustering problem, then any optimization algorithm applied to the proposed variational formulation outputs the same local solution. We then make explicit why the optimality conditions of the proposed $\ell_1$-regularized PageRank problem imply the special termination criterion of APPR, and thus its solution provides the same combinatorial guarantees as in~\cite{ACL06}. 

Although any optimization method applied to our proposed formulation naturally produces the same output, what differentiates between them is their running time. As a result, we present an algorithm based on iterative shrinkage-thresholding algorithm (ISTA)~\cite{bt09} that solves the $\ell_1$-regularized PR problem, while maintaining a running time in the order of the volume of nodes/non-zeros in the optimal solution (i.e., \emph{independent} of the size of the graph). We show that the considered algorithm only requires access to the graph in a localized manner, and hence enjoys similar locality properties as the original APPR. 
%Compared to APPR, we show that ISTA with a \emph{particularly chosen initialization} enjoys a running time which exhibits better dependence on PageRank parameters.

%Let $\mathcal{S}_*$ be the set of non-zero nodes in the solution of the $\ell_1$-regularized PR problem.
%By exploiting strong convexity of the $\ell_1$-regularized PR we prove that the running time of \emph{a particularly initialized} ISTA is $$\mathcal{O}\left(1/(\alpha + \frac{1-\alpha}{2}\lambda_{min}(\mathcal{L}_{\mathcal{S}_*}))\right)$$ where $\alpha$ is the teleportation parameter of the PageRank linear system and $\lambda_{min}(\mathcal{L}_{\mathcal{S}_*})> 0$ also depends on $\alpha$; details are given in Remark \ref{remark:1}. 
%This result is an improvement over the standard analysis for the running time of APPR which has $\mathcal{O}\left(1/\alpha\right)$ dependence on $\alpha$.s
Finally, by taking advantage of the local nature of iterations, we carefully implement the proposed algorithm in C++ and illustrate a few numerical experiments on several large-scale real graphs.

The rest of this paper is organized as follows. Notation used throughout the paper is introduced in Section~\ref{section:preliminaries}. Section~\ref{sec:background} provides a brief introduction to APPR and, in doing so, motivates our intentions in this paper. Our variational formulation is derived in Section~\ref{subsec:modacl}. The application of ISTA for solving this variational formulation is considered in Section~\ref{sec:ista_alg}. This is then followed by numerical simulations on a few real graph data in Section~\ref{sec:examples}. Conclusions and further thoughts are gathered in Section~\ref{sec:conclusion}.

\section{Notation and assumptions}
\label{section:preliminaries}
Throughout the paper, vectors are denoted by bold lowercase letters, e.g., $\qq$, and matrices are denoted by regular upper case letters, e.g., $A$.
The $i^{th}$ coordinate of a vector $\qq$ is denoted by $\qq(i)$ or $[\qq]_{i}$, depending on which is less cumbersome in a given formula. Iteration counter is denoted by $k$ and is placed as subscripts, e.g., $\qq_{k}$ denotes the vector corresponding to $k^{th}$ iteration.The dot-product between two vectors is denoted by $\lin{\pp,\qq} = \pp^{T} \qq$.
The vector of all ones and the vector whose $i^{th}$ coordinate is one and zero elsewhere are denoted by $\ee$ and $\ee_{i}$, respectively.
The square root of a vector is taken component-wise, i.e., $\qq^{1/2}:=[\qq(1)^{1/2},\dots,\qq(n)^{1/2}]$. 

We assume that we are given an undirected graph $\mathcal{G}$ with no self-loops, whose number of nodes and edges 
are denoted by $n$ and $m$, respectively. 

The set of nodes of the graph is denoted by $\V$. By $j\sim i$ we mean that $j$ is a neighbor of $i$ and vice-versa. For a set of nodes $S$, the relation $j\sim S$ indicates that a node $j$ is a neighbor of at least one node in $S$, $\mbox{vol}(S):= \sum_{i\in S} d_i$ and $d_i$ is the number of edges of node $i$, i.e., the degree of node $i$. We reserve $\dd$ to be the vector whose components are degrees of the nodes, i.e., $\dd(i) = d_{i}$.
Matrices $A$ and $D$ denote, respectively, the adjacency matrix and the diagonal degree matrix 
of $\mathcal{G}$. Recall that the $i^{th}$ diagonal element of $D$ is given by $d_i$. 
For $$Q:= D^{-1/2} \left\{D - \frac{1-\alpha}{2}(D + A)\right\}D^{-1/2},$$ we define
\begin{equation}
f(\qq):=\frac{1}{2} \lin{\qq, Q \qq} - \alpha \lin{\bss, D^{-1/2} \qq},
\label{eq:f}
\end{equation}
\textcolor{black}{where $\bss$ is a given distribution over the nodes also known as teleportation distribution}.\
% $f(\qq):=\frac{1}{2}\qq^T Q \qq - \alpha \bss^T D^{-1/2} \qq$.
For $S\subseteq [n]$ where $[n]= \{1,2,\dots,n\}$, let $I_S\in\mathbb{R}^{n\times |S|}$ be 
a $\mathbb{R}^{n\times |S|}$ matrix whose columns are taken from those of the $\mathbb{R}^{n\times n}$ identity matrix indexed by $S$.\ 
Further, we define $\nabla_S f(\qq) := I_S^T \nabla f(\qq)$, $Q_S := I_S^T Q I_S$, and $\dd_{S} := \text{diag}(I_S^T D I_S)$, where ``$\text{diag}(\cdot)$'' extracts the diagonal of the input matrix and returns it as a vector.\
%One can easily see that function $\nabla f$ is block $(1+\alpha)/2$-Lipschitz continuous and the function $f$ is $\alpha$-strongly-convex, both w.r.t. $\ell_2$ norm. (\textcolor{black}{need to improve the strong convexity parameter}).
We also define the support set of a vector $\qq$ as the index set of its non-zero elements, i.e., $\mbox{supp}(\qq):=\{i\in [n] \ | \ \qq(i)\neq 0\}$.
\textcolor{black}{One can easily see that function $\nabla f$ is 1-Lipschitz continuous w.r.t. $\ell_2$ norm,
that is, the largest eigenvalue of $Q$ is smaller or equal to $1$.
To prove this note that $Q = \alpha I  + \frac{1-\alpha}{2}\mathcal{L}$, where $\mathcal{L} = I - D^{-1/2}AD^{-1/2}$ is the symmetric normalized Laplacian matrix.
Using the fact that the largest eigenvalue of $\mathcal{L}$ is bounded by $2$ and the latter definition of $Q$ we obtain the result.
Furthermore, note that this condition implies that $\forall \pp,\qq \in \mathbb{R}^{n}$
\begin{equation*}
\|\nabla f(\pp) - \nabla f(\qq) \|_{2} \le \|\pp - \qq\|_{2},
\end{equation*}
which also implies
\begin{equation*}
f(\pp) \le f(\qq) + \lin{\nabla f(\qq),\pp-\qq} + \frac{1}{2}\|\pp - \qq\|_2^2.
\end{equation*}
}

\section{Background and Motivation}
\label{sec:background}
%in a search engine's results in order to distinguish highly referenced websites  from those that were less popular. 
Suppose $n$ denotes the total number of nodes. \textcolor{black}{A simplified version of PageRank (PR) algorithm}~\cite{page1999pagerank} amounts to computing the stationary solution of 
$$
\pp_{k+1}(j) = \sum_{i\sim j} \pp_k(i)/d_i,
$$
where each node is modeled as a node of a graph, and the components of the vector $\pp \in \mathbb{R}^{n}$ represent the ``popularity'' of these $n$ nodes. Usually the ``popularity'' is encoded as a probability mass distributed over all the nodes, i.e., the vector $\pp$ is like a probability mass function where $\pp \geq \bold 0$ and $\ee^T\pp = 1$ .  
As a result, operationally, the \textcolor{black}{simplified PR algorithm} iteratively transfers probability mass around the graph by adding to a node's assigned probability and taking the equivalent amount from its neighbors. \textcolor{black}{The stationary vector corresponding to this iterative operation is the degrees vector $\dd$}. In Linear Algebra's jargon, \textcolor{black}{the above simplified version of the PR algorithm} amounts to the computation of the principal eigenvector of a large and sparse matrix, $AD^{-1}$, often referred to as transition matrix, i.e.,
$$
AD^{-1} \pp = \pp.
$$

\textcolor{black}{This simplified version of the PR algorithm} has several disadvantages. A particular issue arise when some node is isolated and lacks edges to other nodes, in which case, the above procedure is not well-defined, i.e., the node's degree is zero. This type of nodes are often referred to as ``dangling nodes'' and an elegant way to handle such situations was proposed in \cite{EMT2004}. As a result, for simplicity's sake, we assume that the dangling nodes are dealt with in a proper way and hence, $d_i>0, \; \forall i \in [n]$. 

The second disadvantage is that the convergence to the principal eigenvector of $AD^{-1}$ requires the transition matrix to be aperiodic and irreducible, i.e., the smallest eigenvalue of $AD^{-1}$ is in absolute value less than $1$, and matrix $(AD^{-1})^t$ is component-wise positive for some $t$. 
The former issue can be resolved by considering the lazy random walk matrix, $W = (I + A D^{-1})/2$ instead of $AD^{-1}$, while for the latter, one can consider a convex combination of the form
\begin{equation}\label{eq:convexPPRmat}
\alpha \bss \ee^{T}+ (1-\alpha) W,
\end{equation}
where $\alpha\in(0,1)$ is the ``teleportation'' parameter and $\bss$ is a given distribution over the nodes also known as teleportation distribution. 
\textcolor{black}{The principal eigenvector of matrix \eqref{eq:convexPPRmat} is known as the PR vector~\cite{page1999pagerank}}.
The celebrated PageRank (PR) vector was initially developed in~\cite{page1999pagerank} to rank websites/nodes according to their ``popularity''.

Initially, $\bss$ was set to have uniform probability distribution over all the nodes.
However, ``personalized'' distributions became popular \cite{G15} which assign non-uniform probability mass in favor of certain nodes and, as a result, one seeks to obtain personalized principal eigenvectors of matrix \eqref{eq:convexPPRmat}. 
For example, after arbitrarily ordering the nodes of $\G$, consider an input node, say $i$, and a vector $\bss\in\mathbb{R}^n$ such that $\bss(i)=1$ and zero elsewhere. For a lazy random walk matrix, $W = (I + A D^{-1})/2$, finding the principal eigenvector of \eqref{eq:convexPPRmat} which also satisfies $\ee^T \pp = 1$ and $\pp \ge \bold 0$, is equivalent to the solution of the linear system
\begin{equation}
\label{eq:ppr}
\pp = \alpha \bss + (1-\alpha)W \pp.
\end{equation}
This approach is known as Personalized PageRank (PPR), and in fact, has become the ubiquitous tool for ranking web pages, social and information network analysis, recommendation systems, analysis of biology, neuroscience and physics networks; see~\cite{G15} for an excellent review of PR and PPR as well as their applications. 

Approximate Personalized PageRank (APPR), was first introduced in the seminal work of~\cite{ACL06}. As it appears from its name, APPR is an approximate version of PPR which boils down to approximately solving the linear system~\eqref{eq:ppr} using a particular iterative scheme and a specifically chosen early stopping criterion. 
In fact, it can be shown that APPR's original algorithm is, indeed, an iterative coordinate solver for the linear system~\eqref{eq:ppr}.
%APPR solves approximately the PageRank linear system
%\begin{equation}\label{eq:ppr}
%\pp = \alpha s + (1-\alpha)Wp,
%\end{equation}
%where $\pp$ is the unknown vector
%and $W = (I + AD^{-1})/2$ is a lazy random walk matrix. 
To see this, let us first define the residual vector as $\rr := (I - (1-\alpha)W)\pp - \alpha \bss$.
An iterative coordinate solver applied to \eqref{eq:ppr} updates the current 
approximate solution at iteration $k$ according to $\pp_{k+1} = \pp_k - \rr_k(i) \ee_i$.
%\begin{algorithm}[tb]
%\caption{Coordinate Solver for \eqref{eq:ppr}}\label{acl:cd}
%\begin{algorithmic}[1]
%\STATE \textbf{Initialize}: $\rho>0$, $p_0 = 0$, thus $r_0 = - \alpha \bss$
%\WHILE {$\|D^{-1}r_k\|_\infty > \alpha \rho$} 
%\STATE Choose any coordinate $i$ such that $r_k(i) \le -\alpha d_i\rho$
%\STATE $\pp_{k+1}(i) = p_k(i) - r_k(i)$ 
%\STATE $r_{k+1} = (I - (1-\alpha)W)\pp_{k+1} - \alpha \bss.$
%\STATE $k = k + 1$
%\ENDWHILE
%\STATE \textbf{return} $p_k$
%%\EndProcedure
%\end{algorithmic}
%\end{algorithm}
As a result, the residual vector has the following recursive representation
\begin{equation}
\label{eq:ppr_residual}
\rr_{k+1} = \rr_k - \rr_k(i) \ee_i + \frac{1-\alpha}{2}(I + AD^{-1})\rr_k(i) \ee_i. 
\end{equation}
Algorithm~\ref{acl:cd2} gives an overview of such iterative coordinate solver with a particular stopping criterion. From the definitions of $D$ and $A$, it can easily be seen that Steps $5$, $6$, and $7$ practically implement the recursive relation~\eqref{eq:ppr_residual}.
%\begin{itemize}
%\item[-] For a coordinate $j\neq i$, if for the chosen $i$'th coordinate in step $4$ in Algorithm \ref{acl:cd} we have that $(i,j)\in \mathcal{E}$, then
%$$
%r_{k+1}(j)= r_k(j) + \frac{1-\alpha}{2d_i}r_k(i).
%$$
%\item[-] For $j=i$ we have that
%$$
%r_{k+1}(j)= \frac{1-\alpha}{2}  r_k(i).
%$$
%\item[-] For a coordinate $j\neq i$, if for the chosen $i$'th coordinate in step $4$ in Algorithm \ref{acl:cd} we have that $(i,j)\notin \mathcal{E}$, then
%$$
%r_{k+1}(j)= r_k(j) . 
%$$
%\end{itemize}
\begin{algorithm}[htb]
\caption{Coordinate solver (APPR) for \eqref{eq:ppr}}
\begin{algorithmic}[1]
\STATE \textbf{Initialize}: $\rho>0$, $\pp_0 = \bold{0}$, thus $\rr_0 = - \alpha \bss$
\WHILE {$\|D^{-1} \rr_k\|_\infty > \rho \alpha $} \vspace{0.1cm}
\STATE Choose an $i$ such that $\rr_k(i) < -\alpha d_i\rho$
\STATE $\pp_{k+1}(i) = \pp_k(i) - \rr_k(i)$ 
\STATE $\rr_{k+1}(i) = \frac{1-\alpha}{2}  \rr_k(i)$
\STATE For each $j$ such that $j\sim i$ set
$$
\rr_{k+1}(j) = \rr_k(j)+ \frac{1-\alpha}{2d_i} A_{ij}\rr_k(i)
$$
\STATE For each $j$ such that $j \nsim i$ set $\rr_{k+1}(j)  = \rr_k(j)$
\STATE $k = k + 1$
\ENDWHILE
\STATE \textbf{return} $\pp_k$
%\EndProcedure
\end{algorithmic}
\label{acl:cd2}
\end{algorithm}

Now, by defining $\tilde{\rr}_k := - (1/\alpha) \rr_k$ and replacing $\rr_k$ with $\tilde{\rr}_k$ in Algorithm \ref{acl:cd2} we obtain APPR algorithm in exactly the same form as described in \cite[Section 3]{ACL06}. This indeed shows that APPR is an iterative coordinate solver for the PPR linear system \eqref{eq:ppr}.
%\begin{algorithm}
%\caption{Coordinate solver for \eqref{eq:ppr}}\label{acl:cd3}
%\begin{algorithmic}[1]
%\STATE \textbf{Initialize}: $\rho>0$, $p_0 = 0$, thus $\tilde{r}_0 = - \bss$ \vspace{0.1cm}
%\WHILE {$\|D^{-1}\tilde{r}_k\|_\infty > \rho$} \vspace{0.1cm}
%\STATE Choose any coordinate $i$ where $\tilde{r}_k(i) \ge \rho d_i$\vspace{0.1cm}
%\STATE $\pp_{k+1}(i) = p_k(i) + \alpha \tilde{r}_k(i)$  \vspace{0.1cm}
%\STATE $\tilde{r}_{k+1}(i)= \frac{1-\alpha}{2}  \tilde{r}_{k}(i)$\vspace{0.1cm}
%\STATE For each $j$ such that $(i,j)\in\mathcal{E}$ calculate
%$$
%\tilde{r}_{k+1}(j)= \tilde{r}_{k}(j) + \frac{1-\alpha}{2d_i}\tilde{r}_{k}(i)
%$$
%\STATE For each $j$ such that $(i,j)\notin \mathcal{E}$
%$$
%\tilde{r}_{k+1}(j)= \tilde{r}_{k}(j)
%$$
%\STATE $k = k + 1$
%\ENDWHILE
%\STATE \textbf{return} $p_k$
%%\EndProcedure
%\end{algorithmic}
%\end{algorithm}

%To show that Algorithm \ref{acl:cd2} is equivalent to a coordinate descent algorithm we need to 
%define an optimization problem.
\begin{algorithm}[htb]
\caption{Coordinate descent solver for ``$\min f(\qq)$''}
\begin{algorithmic}[1]
\STATE \textbf{Initialize}: $\rho>0$, $\qq_0 = \bold{0}$, thus $\nabla f(\qq_0) = - \alpha D^{-1/2} \bss$ \vspace{0.1cm}
\WHILE {$\|D^{-1/2}\nabla f(\qq_k)\|_\infty > \rho \alpha $} \vspace{0.1cm}
\STATE Choose an $i$ such that $\nabla_i f(\qq_k) < -\alpha \rho d^{1/2}_i$
\STATE $\qq_{k+1}(i) = \qq_k(i) - \nabla_i f(\qq_k)$  \vspace{0.1cm}
\STATE $\nabla_i f(\qq_{k+1})=\frac{1-\alpha}{2} \nabla_i f(\qq_k)$\vspace{0.1cm}
\STATE For each $j$ such that $j\sim i$ set
$$
\nabla_j f(\qq_{k+1})= \nabla_j f(\qq_k)+ \frac{(1-\alpha)}{2d^{1/2}_id^{1/2}_j} A_{ij}\nabla_i f(\qq_k)
$$
\STATE For each $j$ that $j\nsim i$ set $\nabla_j f(\qq_{k+1}) = \nabla_j f(\qq_k)$
\STATE $k = k + 1$
\ENDWHILE
\STATE \textbf{return} $\pp_k:=D^{1/2}\qq_k$
\end{algorithmic}
\label{acl:cd_ppr}
\end{algorithm}

It is, in fact, easy to see that Algorithm~\ref{acl:cd2} solves the optimization problem  ``$\min f(\qq)$'',
where $f$ is defined as in~\eqref{eq:f}.\ To see this,
%Let us now show that APPR is equivalent to a coordinate descent method for minimizing function $f$.
note that the residual in Algorithm \ref{acl:cd2} can be written in terms of the scaled gradient of function $f$.\
In particular, since
\begin{equation*}
\nabla f(\qq) = D^{-1/2} \left\{D - \frac{1-\alpha}{2}(D + A)\right\}D^{-1/2}\qq - \alpha D^{-1/2} \bss,
\end{equation*} 
we have $D^{1/2} \nabla f(\qq) = \rr$,\ where $\qq := D^{-1/2}\pp$.\
Using $D^{1/2} \nabla f(\qq) = \rr$ we can rewrite Algorithm \ref{acl:cd2} as a coordinate descent method for minimizing $f$ as in Algorithm \ref{acl:cd_ppr}. 

The above simple observation is a motivating factor behind our objective of deriving the exact variational formulation of APPR. However, before delving into the details of this derivation, let us briefly review the combinatorial guarantees of APPR, with respect to graph clustering. This is indeed important in light of our new variational formulation and the proposed algorithm for solving it. In particular, we will show that the optimality condition corresponding to this variational formulation, in fact, implies the special termination criterion of APPR, and hence, the proposed algorithm, upon termination, recovers a cluster with the same combinatorial guarantees as the solution of APPR.

%\subsection{Combinatorial Guarantees of APPR}
%\label{sec:graphguarantees}
Conductance is a widely used concept in graph clustering to measure the quality of a cluster. Loosely speaking, conductance of a cluster is defined as the ratio of its external over internal connectivities. Lower conductance translates to a better cluster since it implies the cluster is better connected internally than externally. More specifically, let $w_{ij}$ be the weight of the edge between two neighbor nodes $i\sim j$.\
We define the conductance of a subset of nodes $S\subset \V$ as
$$
\Phi(S) := \frac{\displaystyle\sum_{i \in S} \displaystyle\sum_{j\in \V\backslash S,j\sim i} w_{ij}}{\min \left(\mbox{vol}(S) , \mbox{vol}(\V\backslash S) \right)}
$$
and the minimum-conductance of a given graph $\G$ as 
\begin{equation}\label{eq:phig}
\Phi (\G):= \min_{S\subset \mathcal{V}} \Phi(S).
\end{equation}

%In what follows, we discuss the performance of APPR for local graph clustering using conductance. 
Given a target cluster $C$ with conductance $\Phi(C) \le \Omega(\varphi^2/\log m)$ and $\alpha$ set properly according to $\varphi$, a particular rounding algorithm is applied to the output of APPR which determines a set of nodes in the graph with conductance of at most $\varphi$. More precisely, let $\pp_k$ be the output of APPR with input value $\alpha$ and let $\rr_k$ be the residual of \eqref{eq:ppr}.\
According to~\cite[Theorem 5]{ACL06}, the output of APPR can be used as an input to a rounding procedure (see~\cite[Section 2.2 ]{ACL06})
to produce clusters of low-conductance.\
The rounding procedure sorts the indices in $\text{supp}(\pp_k)$ in decreasing order according to the values of the components of $D^{-1}\pp_k$. Let $i_1,i_2,\dots,i_{|H_k|}$ be the sorted indices, where $H_k = \text{supp}(\pp_k)$.\ Using the sorted indices, the rounding procedure generates a collection of sets $\mathcal{S}_{j}:=\{i_1,i_2,\dots,i_j\}$ for each $j\in \{1,2,\dots,|H_k|\}$.\
Provided that there exists a subset of nodes, $C$, such that $\Phi(C) \le \alpha/10$, $\text{vol}(C) \le 2\text{vol}(\G)/3$, $\bss$ is initialized within nodes in $C_\alpha$, where $C_\alpha \subseteq C$ satisfies $\text{vol}(C_\alpha) \ge \text{vol}(C)/2$,
and $\rho = {1}/{(10 \text{vol}(C))}$ 
then~\cite[Theorem 5]{ACL06} implies that $$\min\limits_{j\in \{1,2,\dots,|H_k|\}}\Phi(\mathcal{S}_j)\leq \sqrt{135 \log(m) \alpha}.$$ 
This result is a local analogue of the Cheeger inequality \cite{Cheeger69_bound} for PageRank vectors.
%Theorem $5$ in \cite{ACL06} suggests setting $\alpha = 10 \Phi(\G)$ which gives us
%$\min\limits_{j\in \{1,2,\dots,|H_k|\}} \Phi(\mathcal{S}_j)  \le \sqrt{1350 \Phi(\G) \log(m)}$.
%Similar guarantees can be proved if we replace APPR with Algorithm \ref{acl:l1pg}.

An undesirable side-effect of this rounding procedure is the lack of a lower bound on the volume of the 
output cluster. This, in particular, implies that it is possible to find a very small cluster. 
As a remedy,~\cite[Section 6]{ACL06} introduces PageRank-Nibble procedure.\ 
%In the remaining section we describe the main result about PageRank-Nibble from Section $6$ in \cite{ACL06}, and we refer the reader to \cite{ACL06} for a more detailed discussion of the algorithm.
Let $\phi\in[0,1]$ be a parameter and assume that there exists $C\subset \V$ such that $\text{vol}(C) \le \text{vol}(\G)/2$ and $\Phi(C) \le \phi^2/(22500\log^2 (100m))$. PageRank-Nibble makes only a single call to APPR and uses its output to produce the rounded sets as before. However~\cite[Theorem 7]{ACL06} suggests that if APPR is initialized with $\alpha = \phi^2/(225\log (100 m^{1/2}))$ and $\bss$ is set in $C_\alpha$, then there exists some $b\in [1, \ceil{\log m}]$ such that if $\rho \le (2^b 48 \ceil{\log m})^{-1}$, at least one set $\mathcal{S}_{j}$ satisfies 
$\Phi(\mathcal{S}_{j}) \le \phi$, $2^{b-1} < \mbox{vol}(\mathcal{S}_{j}) < 2\mbox{vol}(\G)/3$ and $\mbox{vol}(\mathcal{S}_{j}\cap C) > 2^{b-2}$.\
%PageRank-Nibble requires a single call to Algorithm \ref{acl:cd2}.\ 

%\begin{theorem} \label{thm:pageranknibble} 
%Let $\phi\in[0,1]$, $\alpha = \phi^2/(225\log (100 m^{1/2}))$ and $\bss$ is set within a subset of nodes $C_\alpha\subseteq C$, 
%where $C$ satisfies $\mbox{vol}(C) \le \mbox{vol}(\G)/2$ and $\Phi(C) \le \phi^2/(22500\log^2 (100m))$.\
%There exists $b\in [1, \ceil{\log m}]$ such that if $\rho \le (2^b 48 \ceil{\log m})^{-1}$ then 
%Algorithm \ref{acl:l1pg} with $\rho'={\rho}/{(1+\epsilon)}$\
%can be used in PageRank-Nibble to find a set $S$ that satisfies $\Phi(S) \le \phi$, $2^{b-1} < \mbox{vol}(S) < 2\mbox{vol}(\G)/3$ and $\mbox{vol}(S\cap C) > 2^{b-2}$.
%\end{theorem}

\section{Variational Formulation}
\label{subsec:modacl}
In this section we set out to derive the variational formulation characterizing APPR and discuss how we can view the approximate solution of~\eqref{eq:ppr} as the optimal solution of an $\ell_1$-regularized problem.
%In Section~\ref{sec:background}, we showed that Algorithm \ref{acl:cd_ppr} is indeed equivalent to the coordinate descent interpretation of APPR Algorithm \ref{acl:cd2}. 

A key observation which helps us derive the sought-after variational formulation is given by the following lemma. In particular, Lemma~\ref{lemma:aclpossol} shows that the iterates generated by Algorithm \ref{acl:cd_ppr} with a particular initialization, have an interesting property, in that they all satisfy $\nabla f(\qq_k) \le 0$ $\forall k$. 

\begin{lemma}
\label{lemma:aclpossol}
If Algorithm \ref{acl:cd_ppr} is initialized with $\qq_0=0$ and $\bss\ge0$, then $\qq_{k+1}\ge \qq_k$ and $\nabla f(\qq_k) \le 0$ $\forall k$.
\end{lemma}
\begin{proof}
We will prove this statement by induction. Let us assume that at the $k^{th}$ iteration we have $\qq_k\ge 0$ and $\nabla f(\qq_k)\le 0$.
Further, let assume that there exists coordinate $i$ such that $\nabla_i f(\qq_k) < -\rho \alpha d^{1/2}_i$, otherwise, the termination criterion 
is satisfied. Algorithm \ref{acl:cd_ppr} chooses one coordinate which satisfies $\nabla_i f(\qq_k) < -\rho \alpha d^{1/2}_i$. Then from Step $4$ of Algorithm
\ref{acl:cd_ppr} we have that $\qq_{k+1} \ge \qq_k$. Moreover, from Steps $5$, $6$, and $7$, it follows that $\nabla_i f(\qq_k) < \nabla_i f(\qq_{k+1}) < 0$, 
$\nabla_j f(\qq_{k+1}) < \nabla_j f(\qq_k) \le 0$ for each $j$ such that $i\sim j$ and $\nabla_j f(\qq_{k+1}) = \nabla_j f(\qq_k) \le 0$ 
for each $j$ such that $i\nsim j$. Hence, $\nabla f(\qq_{k+1}) \le 0$. Let $\qq_0=0$ and $\bss\ge 0$.
Then $\nabla f(\qq_0) = -\alpha \bss \le 0$. We conclude that $\qq_{k+1} \ge \qq_k\ge 0$ and $\nabla f(\qq_k) \le 0$ $\forall k$.
\end{proof}

On the one hand, as argued in Section~\ref{sec:background}, Algorithm \ref{acl:cd_ppr} is equivalent to the coordinate descent interpretation of APPR. On the other, Algorithm~\ref{acl:cd_ppr} terminates when 
\begin{equation}\label{eq:termcritacl}
\|D^{-1/2}\nabla f(\qq_k)\|_\infty \le \rho \alpha,
\end{equation}
which, since by Lemma \ref{lemma:aclpossol} the gradient components at every iteration are all non-positive, is equivalent to
\begin{equation}\label{acltcond}
\nabla_i f(\qq_k) \ge -\rho \alpha d^{1/2}_i\ \forall i.
\end{equation}
Interestingly, the termination criterion \eqref{acltcond} is related to the first-order optimality conditions of the following $\ell_1$-regularized problem
\begin{equation}\label{l1pg}
\mbox{$\ell_1$-reg. PR: } \boxed{\mbox{minimize} \quad \psi(\qq):= \rho \alpha \|D^{1/2}\qq\|_1 + f(\qq).}
\end{equation}
%Problem \eqref{l1pg} has a unique optimal solution because the quadratic matrix $Q$ in $f$ is symmetric positive definite. 
Let $\qq_{*}$ denote the optimal solution of~\eqref{l1pg}.
The first-order optimality conditions of \eqref{l1pg} can be written as
\begin{equation}
\nabla_i f(\qq_*) = \begin{cases}
			  -\rho\alpha d^{1/2}_i & \mbox{if } \qq_*(i) > 0 \\
			  \ \ \rho\alpha d^{1/2}_i & \mbox{if } \qq_*(i) < 0 \\
			  \in \rho\alpha d^{1/2}_i [-1,1] & \mbox{if } \qq_*(i) = 0.
			  \end{cases} 
\end{equation}
Theorem \ref{thm:reducedacl}, below, shows that the solution of~\eqref{l1pg} has the property that $\qq_*\ge 0$. Therefore, 
the optimality conditions of problem \eqref{l1pg} are equivalent to 
\begin{equation}\label{optcondl1pg}
\nabla_i f(\qq_*) = \begin{cases}
			  -\rho\alpha d^{1/2}_i & \mbox{if } \qq_*(i) > 0 \\
			  \in \rho\alpha d^{1/2}_i [-1,0] & \mbox{if } \qq_*(i) = 0.
			  \end{cases} 
\end{equation}

%We make a few remarks regarding the formulation~\eqref{l1pg} and its termination criteria~\eqref{optcondl1pg}.
The formulation~\eqref{l1pg} is indeed a variational characterization of the APPR procedure as described by its coordinate descent representation in Algorithm~\ref{acl:cd_ppr}. However, notice that the optimality conditions \eqref{optcondl1pg} imply the termination criterion \eqref{acltcond} of APPR, but the converse is not necessarily true. 
This is because~\eqref{acltcond} does not distinguish between positive and zero components of $\qq_*$. 
Moreover, depending on which coordinate is chosen at every iteration, APPR can yield a different output on multiple runs. In other words, the output solution depends completely on the setting of the algorithm. In contrast, $\ell_1$-regularized PR formulation~\eqref{l1pg} decouples the locality/sparsity  of the solution from properties of the algorithm, i.e., which nodes are chosen at every iteration. More specifically, if there exists a good local cluster, then \textit{any optimization algorithm applied to $\ell_1$-regularized PR obtains the same solution, and the differences merely boil down to running time and locality as opposed to the actual output solution}. \textcolor{black}{Note that in practice algorithms solve approximately the $\ell_1$-regularized PR, therefore, small differences might exist among solutions of different algorithms. However, the longer that any convergent algorithm is run the closer 
its solution will be to the optimal solution of the $\ell_1$-regularized PR problem.}

The proposed optimization formulation \eqref{l1pg} is motivated by~\cite[Theorem 3]{GM14_ICML}.\ However, by drawing a clear connection between the termination criterion of APPR,~\eqref{acltcond}, and the first-order optimality conditions of $\ell_1$-regularized PR,~\eqref{optcondl1pg}, we get a much simpler formulation than the one presented in \cite{GM14_ICML}.
In particular, unlike the formulation of \cite{GM14_ICML}, problem \eqref{l1pg} does not require any additional tuning parameters other than the ones used for APPR, nor does it introduce any constraints, such as non-negativity. 
More importantly, the formulation in \cite{GM14_ICML} only implies the sparsity of the final solution as opposed to the intermediate iterates produced by any iterative procedure applied to solve the corresponding optimization problem. In sharp contrast, in Section~\ref{sec:ista_alg}, we will show that the application of properly initialized ISTA to our formulation~\eqref{l1pg} maintains sparsity for all generated iterates, a property which is crucial to obtaining a local algorithm.
%Our clean l1-reg. problem is the reason that we can leverage a new class of 
%optimization algorithms.
%In addition we show in Section \ref{section:cutgu} that by solving problem \eqref{l1pg}, one can theoretically achieve the same worst-case local graph clustering guarantees as APPR.

\section{Algorithm}
\label{sec:ista_alg}
As mentioned before, an advantage of the variational formulation~\eqref{l1pg} is that it decouples the properties of the obtained solution from the applied algorithm. This allows for application of any optimization algorithm. However, among all options, we need to find methods that, like APPR, enjoy \emph{locality} properties, in that they only require access to small portion of the graph. In doing so, in this section, we investigate the application of ISTA for solving~\eqref{l1pg} and study its theoretical properties such as locality and running time. The adaptation of ISTA to our particular problem is depicted in Algorithm~\ref{acl:l1pg}. 

\textcolor{black}{The main computational advantage of APPR is that, APPR never requires access to the entire graph and iterations are performed efficiently which makes the application of APPR very appealing for modern large graphs.\ }
Interestingly, we  now show that Algorithm~\ref{acl:l1pg}, which incorporates a presumably global optimization routine such as ISTA, exhibits this desired locality property while inheriting the fast convergence properties of ISTA.

\begin{algorithm}[htb]
\caption{ISTA-equivalent solver for \eqref{l1pg}}
\label{acl:l1pg}
\begin{algorithmic}[1]
\STATE \textbf{Initialize}: $\epsilon\in(0,1)$,  $\alpha>0$, $\qq_0 = 0$, $\rho > 0$, $\bss$ such that $\lin{\ee,\bss} = 1$ and $\bss\ge \bold 0$, set $\nabla f(\qq_0) = - \alpha D^{-1/2}\bss$. \vspace{0.1cm}
%Moreover, choose $\rho$ such that $\|D^{-1/2}\nabla f(q_0)\|_\infty > \rho \alpha$.\vspace{0.1cm}
\WHILE {$\|D^{-1/2}\nabla f(\qq_k)\|_\infty > (1+\epsilon)\rho \alpha $} \vspace{0.1cm}
\STATE Set $S_k:=\{i\in [n] \ | \ \qq_k(i) - \nabla_i f(\qq_k) \ge \rho \alpha d_i^{1/2}\}$\vspace{0.1cm}
%\STATE Choose $S_k^\tau \subseteq S_k$ such that $|S_k^\tau| = \min (\tau, |S_k|)$
\STATE $\Delta \qq_k := - (\nabla_{S_k} f(\qq_k) + \rho \alpha \dd^{1/2}_{S_k})$ and $\qq_{k+1}(S_k) = \qq_k(S_k) + \Delta \qq_k$
\STATE For each $i\in S_k$ set
			\begin{align*}
	    \nabla_{i} f(\qq_{k+1})  = -\rho \alpha d_{i}^{1/2}- \frac{1-\alpha}{2} [I_{S_{k}} \Delta \qq_{k}]_{i} -\frac{1-\alpha}{2d_i^{1/2}}\displaystyle\sum_{l	\sim i, l\in S_k} \frac{A_{il}[I_{S_k}\Delta \qq_k]_l}{d_l^{1/2}}
	    \end{align*}
	    %\begin{align*}
	    %\nabla_{i} f(q_{k+1})  = \nabla_i f(q_k) +  \frac{1+\alpha}{2} [I_{S} \Delta q_{k}]_i - \frac{1-\alpha}{2d_i^{1/2}}\displaystyle\sum_{l\sim i, l\in S} \frac{A_{i,l}[I_{S}\Delta q_k]_l}{d_l^{1/2}})
	    %\end{align*}
\STATE For each $j\notin S_k$ such that $j\sim S_k$ set
\begin{align*}
\nabla_j f(\qq_{k+1})  = \nabla_j f(\qq_k) - \frac{1-\alpha}{2d_j^{1/2}}\displaystyle\sum_{l\sim j, l\in S_k} \frac{A_{jl}[I_{S_k}\Delta \qq_k]_l}{d_l^{1/2}}
\end{align*}
\STATE For each $j\notin S_k$ such that $j\nsim S_k$ set
$$
\nabla_j f(\qq_{k+1}) = \nabla_j f(\qq_k)
$$
\STATE $k = k + 1$
\ENDWHILE
\STATE \textbf{return} $\pp_k:= D^{1/2}\qq_k$
\end{algorithmic}
\end{algorithm}

%First let us explain why Algorithm~\ref{acl:l1pg} has running time that depends only on the number of nonzeros of $\qq^*$.
Theorem~\ref{thm:reducedacl} shows the equivalence between Algorithm~\ref{acl:l1pg} and ISTA, and more importantly, establishes the desired locality property. In particular, part~\ref{thm:ii} of Theorem~\ref{thm:reducedacl} states that if Algorithm~\ref{acl:l1pg} is initialized properly, then despite the fact that the set $S_{k}$ changes at every iteration (Step 3 of Algorithm~\ref{acl:l1pg}), its size, $|S_{k}|$, indeed \emph{never grows larger} than the total number of non-zeros of the optimal solution. As such, in the worst case where one might update all the coordinates in $S_{k}$ at every iteration, the per-iteration cost depends only on the sparsity of the final solution vector, as opposed to the size of the full graph.
%This is proved in Theorem~\ref{thm:reducedacl} by initializing Algorithm~\ref{acl:l1pg} with the zero vector. 
%Furthermore, we show in Theorem~\ref{thm:reducedacl} that each iteration of Algorithm~\ref{acl:l1pg} is indeed equivalent to each iteration ISTA.

\begin{theorem}
\label{thm:reducedacl}
Let $\qq_*$ be the optimal solution of \eqref{l1pg} and consider $\rho > 0$ and a vector $\bss \geq 0 $ such that $\lin{\ee, \bss} = 1$ and $\|\bss\|_\infty \ge \rho$. Algorithm \ref{acl:l1pg} has the following properties.
\begin{enumerate}[label = (\roman*)]
\item \label{thm:i} Algorithm \ref{acl:l1pg} is equivalent to ISTA in \cite{bt09},
\item \label{thm:ii_2} \textcolor{black}{$S_k \subseteq S_{k+1} \subseteq \mbox{supp}(\qq_*)$ $\forall k$},
\item \label{thm:ii} $|S_k| \le |S_{k+1}| \le |\mbox{supp}(\qq_*)|$, \; $\forall k$,
\item \label{thm:iii} $0 \leq \qq_k \leq \qq_{k+1}$, \; $\forall k$, \textcolor{black}{which implies that $\qq_* \geq 0$, since $\qq_k\to \qq_* $ as $k\to\infty$.}
\item \label{thm:iv} $\nabla f (\qq_k) \le 0$, and moreover $\nabla_i f(\qq_k) \le -\rho \alpha d_i^{1/2}$ $\forall i\in S_k$ and $\nabla_i f(\qq_k) > -\rho \alpha d_i^{1/2}$ $\forall i\in [n]\backslash S_k$ $\forall k$.
\end{enumerate}
\end{theorem}
\begin{proof}
Define 
\begin{align*}
\tilde{f}(\qq; \qq_{k}) &:=  f(\qq_{k}) + \lin{\qq-\qq_{k}, \nabla f(\qq_{k})} + \frac{1}{2} \|\qq - \qq_{k}\|_2^2, \\
\tilde{\psi}(\qq; \qq_{k}) &:= \rho \alpha \|D^{1/2} \qq \|_1 + \tilde{f}(\qq; \qq_{k}).
\end{align*}
It is easy to see that
$$
\arg \min_{\qq} \tilde{\psi}(\qq; \qq_{k}) = \arg \min_{\qq} \rho \alpha \|D^{1/2} \qq \|_1 + \frac{1}{2} \| \qq - (\qq_{k} - \nabla f(\qq_{k}) )\|_{2}^{2},
$$
and hence 
$$
\qq(i) = \text{\textbf{prox}}_{\rho \alpha d_{i}^{1/2} \|.\|_{1}} \left(\qq_{k}(i) - \nabla_{i} f(\qq_{k})\right),
$$
where \text{\textbf{prox}} is the proximal operator~\cite{parikh2013proximal}.\ Now let us define the sets 
\begin{align}\label{eq:sets}\nonumber
S_k   & := \{i\in [n] \ | \ \qq_k(i) - \nabla_i f(\qq_k) \ge \rho \alpha d_i^{1/2}\},\\
\widehat{S}_k & := \{i\in [n] \ | \ -\rho \alpha d_i^{1/2} < \qq_k(i) -  \nabla_i f(\qq_k) <  \rho \alpha d_i^{1/2}\},\\\nonumber
\widetilde{S}_k & := \{i\in [n] \ | \ \qq_k(i) - \nabla_i f(\qq_k) \leq  - \rho \alpha d_i^{1/2}\}.
\end{align}
For convenience, below, we rewrite ISTA from~\cite{bt09}. 
\begin{algorithm}[htb]
\caption{ISTA for \eqref{l1pg}}\label{acl:ista}
\begin{algorithmic}[1]
\STATE \textbf{Initialize}: $\rho>0$, $\qq_0 = 0$, thus $\nabla f(\qq_0) = - \alpha D^{-1/2}\bss$ 
\WHILE {termination criteria are not satisfied} \vspace{0.1cm}
\STATE  $\qq_{k+1}(i) = \text{\textbf{prox}}_{\rho \alpha d_{i}^{1/2} \|.\|_{1}} \left(\qq_{k}(i) - \nabla_{i} f(\qq_{k})\right), \; \forall i$, whose closed-form solution is given by
\begin{equation*}
\qq_{k+1}(i) = \begin{cases}
			  \qq_k(i)  - (\nabla_i f(\qq_k) + \rho \alpha d_i^{1/2}) & \mbox{if } i\in S_k \\
			  \qq_k(i)  - (\nabla_i f(\qq_k) - \rho \alpha d_i^{1/2})  & \mbox{if } i\in \widetilde{S}_k \\
			  0 & \mbox{if } i\in \widehat{S}_k.
		  \end{cases} 
\end{equation*} 
\STATE Calculate new gradient $\nabla f(\qq_{k+1})$.
\STATE $k = k + 1$
\ENDWHILE
\STATE \textbf{return} $\pp_k:= D^{1/2}\qq_k$
\end{algorithmic}
\end{algorithm}
To show that Algorithms~\ref{acl:l1pg} and~\ref{acl:ista} are equivalent, it suffices to show that $\widetilde{S}_{k} = \emptyset, \forall k$. We will prove the result by induction. Let us assume that at iteration $k$ we have a $\qq_k\ge 0$, $\nabla f(\qq_k) \le 0$ and $\nabla_i f(\qq_k) \le -\rho \alpha d_i^{1/2}$ $\forall i\in S_k$. As a result of the first two assumptions, we have \textcolor{black}{$\widetilde{S}_k=\emptyset$ and $S_{k} \cup \widehat{S}_k = [n]$}. Hence, Step $3$ of ISTA Algorithm~\ref{acl:ista} can be simplified as 
\begin{equation}\label{ista-step-red}
\qq_{k+1}(i) = 
		\begin{cases}
			  \qq_k(i) - (\nabla_i f(\qq_k) + \rho \alpha d_i^{1/2} ) & \mbox{if } i\in S_k \\
			  0 & \mbox{if } i\in \widehat{S}_{k}
		  \end{cases}.
\end{equation} 
Define $\Delta \qq_{k} := - I_{S_{k}}^{T}\big(\nabla f(\qq_k) + \rho \alpha D^{1/2} \ee \big)$, where $I_{S_k}$ is defined in Section \ref{section:preliminaries}.
Consequently, at iteration $k$, the new gradient components are updated as follows
\begin{equation}\label{ista-step-grad}
\nabla_{i} f(\qq_{k+1}) = 
		\begin{cases}
		-\rho\alpha d_i^{1/2}  - \frac{1-\alpha}{2}[I_{S_k}\Delta \qq_k]_i - \frac{1-\alpha}{2 d_i^{1/2}}\displaystyle\sum_{l\sim i, l\in S_k} \frac{A_{il}[I_{S_k}\Delta \qq_k]_l}{d_l^{1/2}}, \quad i \in S_{k} \\
			  %\nabla_i f(\qq_k) +  \frac{1+\alpha}{2} [I_{S_k}\Delta \qq_{k}]_{i} - \frac{1-\alpha}{2 d_i^{1/2}}\displaystyle\sum_{l\sim i, l\in S_k} \frac{A_{i,l}[I_{S_k}\Delta \qq_k]_l}{d_l^{1/2}}, \quad i \in S_{k} \\
			  \nabla_i f(\qq_k) - \frac{1-\alpha}{2d_i^{1/2}}\displaystyle\sum_{l\sim i, l\in S_k} \frac{A_{il}[I_{S_k}\Delta \qq_k]_l}{d_l^{1/2}}, \quad \quad \quad \; \; \; \;i\in \widehat{S}_k \text{ and } i \sim S_k\\
				\nabla_i f(\qq_k), \quad \quad \quad \quad \quad  \quad \quad \; \; \quad \; \quad \quad \quad \quad \quad \quad \; i \in \widehat{S}_k \text{ and } i \nsim S_k,
		  \end{cases}
\end{equation} 
where $A$ is the adjacency matrix of the given graph. \textcolor{black}{Equation \eqref{ista-step-grad} is obtained by using 
$\nabla f(\qq_{k+1}) = \nabla f(\qq_{k}) - I_{S_k}\Delta \qq_k - \frac{1-\alpha}{2} I_{S_k}\Delta \qq_k - \frac{1-\alpha}{2}D^{-1/2}AD^{-1/2}I_{S_k}\Delta \qq_k$ and the definition of $\Delta \qq_k$}. By induction hypothesis and noticing that $\Delta \qq_{k} \geq 0$ and $A_{i,l} \geq 0, \forall i,l$, it is easy to see that by~\eqref{ista-step-red}, we have $\qq_{k+1} \geq 0$, and by~\eqref{ista-step-grad}, we get $\nabla f(\qq_{k+1}) \le 0$. Hence, it follows that $\widetilde{S}_{k+1} = \emptyset$. In addition, for any $i \in S_{k}$, we get $\nabla_{i} f(\qq_{k+1}) \leq -\rho \alpha d^{1/2}_{i}$ and, as such, $i \in S_{k+1}$. In other words, once an index $i$ enters the set $S_{k}$ at iteration $k$, it will continue to stay in that set for all subsequent iterations, and so we always have $\qq_{k+1}(i) \geq \qq_{k}(i)$.\ 
\textcolor{black}{As a result we obtain $S_k \subseteq S_{k+1}$ and $|S_k| \le |S_{k+1}|$}.
\textcolor{black}{The only indices entering $S_{k+1}$ are those from $\widehat{S}_{k}$ that are also neighbors of $S_{k}$. To prove this use that $\widetilde{S}_k = \emptyset$ $\forall k$, therefore the only coordinates that can enter in $S_k$ come from $\widehat{S}_k$.  In addition from \eqref{ista-step-red} we have that $[\qq_k]_i = 0$ $\forall i \in \widehat{S}_k$ and from \eqref{ista-step-grad} we have that 
neighbors of $S_k$ that are also in $\widehat{S}_k$ get their partial derivatives updated. Therefore, using the definition of $S_k$ in \eqref{eq:sets} only the neighbors of $S_k$ that are also in $\widehat{S}_k$ might enter $S_k$, since the rest of the coordinates in $i\in \widehat{S}_k$ have $[\qq_k]_i=0$ and also do not get their partial derivatives updated}. In this case, suppose that $i \in \widehat{S}_{k} \cap S_{k+1}$. By~\eqref{ista-step-red}, we have $\qq_{k+1}(i) = 0$, which combined with the definition of $S_{k+1}$, yields $\nabla_i f(\qq_{k+1}) \le -\rho \alpha d_i^{1/2}$. As a result, we have $\nabla_i f(\qq_{k+1}) \le -\rho \alpha d_i^{1/2}, \forall i\in S_{k+1}$. All is left to do is to start the iterations with the proper initial conditions, so that the base case of the induction holds. Set $\rho$ small enough that $\|\bss\|_\infty \ge \rho$. Now since $\bss\ge 0$, by choosing $\qq_0=0$, we have that $\nabla f(\qq_0) = -\alpha D^{-1/2} \bss \le 0$ and $\nabla_i f(\qq_0) \le -\rho \alpha d_i^{1/2}$ $\forall i \in S_0$. In addition, such a choice of $\qq_{0}$,~\eqref{ista-step-red} as well as the decreasing nature of $\widehat{S}_{k}$ imply that $\qq_{k+1} \geq \qq_{k}, \forall k$.  \textcolor{black}{Since $\qq_{k+1} \ge \qq_k$ $\forall k$ and $\qq_k \to \qq_*$ then Algorithm \ref{acl:l1pg} will update only coordinates that are in $\mbox{supp}(\qq_*)$.
To prove this note that if a coordinate in $\qq_k$ becomes positive it will remain positive because $\qq_{k+1} \ge \qq_k$. Since $\qq_k \to \qq_*$ it must be that only coordinates in $\mbox{supp}(\qq_*)$ 
will become positive in $\qq_k$ for some $k$. Thus, we have that $S_k \subseteq \mbox{supp}(\qq_*)$ and $|S_k|\le |\mbox{supp}(\qq_*)|$ $\forall k$.} 
\textcolor{black}{Finally, notice that $\nabla_i f(\qq_k) > -\rho \alpha d_i^{1/2}$ $\forall i\in [n]\backslash S_k$ $\forall k$. This can be proved by using $[n]\backslash S_k = \widehat{S}_k \cup \widetilde{S}_k$, $\widetilde{S}_k=\emptyset$, $\qq_k\ge 0$ $\forall k$
and using the definition of $\widetilde{S}_k$ in \eqref{eq:sets}.}
\end{proof}

Let 
\begin{equation}
\mathcal{S}_{*}:= \mbox{supp}(\qq_*),
\label{eq:opt_support}
\end{equation}
be the support of the optimal solution. In the following theorem, we give an upper bound for $\text{vol}(\mathcal{S}_*)$ which is, in turn, used in Theorem \ref{thm:aclanalysis}
to derive the worst-case running time of Algorithm \ref{acl:l1pg}.
\begin{theorem}\label{thm:volume}
We have that $\mbox{vol}(\mathcal{S}_*) \le \|\bss\|_1/\rho$, where $\rho$ is the regularization parameter of the $\ell_1$-regularized PageRank \eqref{l1pg}.
\end{theorem}
\begin{proof}
From~\ref{thm:iv} in Theorem \ref{thm:reducedacl} we have that $\nabla_i f(\qq_k) \le -\rho \alpha d_i^{1/2}$ $\forall i\in S_k$ for any iteration $k$. Multiplying both sides of the latter by $-d_i^{1/2}$ and summing over all nodes in $S_k$ yields
$$
\sum_{i\in S_k} -d_i^{1/2}\nabla_i f(\qq_k) \ge \rho \alpha \mbox{vol}(S_k),
$$
which implies that
%From~\ref{thm:iv} in Theorem \ref{thm:reducedacl} we have that $\nabla_i f(\qq_k) \le 0$ $\forall i$, thus 
\begin{equation}\label{eq:vol_bd_k}
\|D^{1/2} \nabla f(\qq_k) \|_1 \ge \rho \alpha \mbox{vol}(S_k).
\end{equation}
We will now prove that $\|D^{1/2} \nabla f(\qq_k) \|_1$ decreases monotonically as $k$ increases.
From Step 4 of Algorithm~\ref{acl:l1pg}, we have $\qq_{k+1} = \qq_{k} + I_{S_{k}} \Delta \qq_{k}$. As a result, from~\eqref{eq:f}, it follows that
\begin{align*}
\nabla f(\qq_{k+1}) &= Q \qq_{k+1} - \alpha D^{-1/2} \bss \\
& = Q \qq_{k} + Q I_{S_{k}} \Delta \qq_{k} - \alpha D^{-1/2} \bss \\
&= \nabla f(\qq_{k}) + Q I_{S_{k}} \Delta \qq_{k} \\
&= \nabla f(\qq_{k}) + \left(\alpha I  + \frac{(1-\alpha)}{2} \left( I - D^{-1/2} A D^{-1/2} \right)\right) I_{S_{k}} \Delta \qq_{k}.
\end{align*}
\textcolor{black}{In the last inequality we used $Q = I - \frac{1-\alpha}{2}(I + D^{-1/2}AD^{-1/2}) = I + \frac{1-\alpha}{2}I - \frac{1-\alpha}{2}I - \frac{1-\alpha}{2}(I + D^{-1/2}AD^{-1/2}) = \alpha I  + \frac{(1-\alpha)}{2} \left( I - D^{-1/2} A D^{-1/2} \right)$.}
Hence, we get
\begin{align*}
D^{1/2} \nabla f(\qq_{k+1}) = D^{1/2} \nabla f(\qq_{k}) + \alpha D^{1/2}  I_{S_{k}} \Delta \qq_{k}  + \frac{(1-\alpha)}{2} ( D  - A ) D^{-1/2} I_{S_{k}} \Delta \qq_{k},
\end{align*}
which implies
\begin{align*}
\ee^{T} D^{1/2} \nabla f(\qq_{k+1}) &= \ee^{T} D^{1/2} \nabla f(\qq_{k}) + \alpha \ee^{T} D^{1/2}  I_{S_{k}} \Delta \qq_{k}  \\
& \quad \quad \quad \quad \quad \quad \quad \quad + \frac{(1-\alpha)}{2} \ee^{T} ( D  - A ) D^{-1/2} I_{S_{k}} \Delta \qq_{k} \\
&= \ee^{T} D^{1/2} \nabla f(\qq_{k}) + \alpha \ee^{T} D^{1/2}  I_{S_{k}} \Delta \qq_{k},
\end{align*}
where for the latter equality, we used the fact that $(D - A) \ee = \bold 0$.
From the proof of Theorem \ref{thm:reducedacl} we have that $\nabla f(\qq_k) \le \bold 0$ and $\Delta \qq_k \ge \bold 0$ $\forall k$. Hence, the last equality implies that
$$
\|D^{1/2}\nabla f(\qq_{k+1}) \|_1 \le \|D^{1/2}\nabla f(\qq_{k}) \|_1.
$$
Using the above inequality and $D^{1/2}\nabla f(\qq_{0}) = -\alpha \bss$ in \eqref{eq:vol_bd_k} we get 
$$
\|s\|_1 \ge \rho \mbox{vol}(S_k) \ \forall k.
$$
Since $S_k \to \mathcal{S}_*$ as $k\to \infty$ then $\|s\|_1 \ge \rho \mbox{vol}(\mathcal{S}_*)$. \textcolor{black}{To prove this use the fact that Algorithm \ref{acl:l1pg} is a convergent algorithm. Therefore, as Algorithm \ref{acl:l1pg} converges to the optimal solution $\qq_*$
then the set $S_k$ converges to $\mathcal{S}_*$, i.e., $S_k$ consists of the same elements as $\mathcal{S}_*$, thus inequality $\|s\|_1 \ge \rho \mbox{vol}(S_k) \ \forall k$ holds for $\mathcal{S}_*$ as well, i.e., $\|s\|_1 \ge \rho \mbox{vol}(\mathcal{S}_*)$.}
\end{proof}

%The following remark provides the strong convexity parameter of $f$ which is used in Theorem \ref{thm:aclanalysis} to bound the worst-case running time of 
%Algorithm \ref{acl:l1pg}.
%
%\begin{remark}\label{remark:1}
We are now ready to derive the overall iteration complexity and the total running time of Algorithm~\ref{acl:l1pg}. For this, we will make use of strong convexity of $f$ in \eqref{eq:f}. It is easy to see that $f$ is $\alpha$-strongly convex. Indeed, $Q$ in~\eqref{eq:f} can be rewritten as $Q = \alpha I  + {(1-\alpha)}\mathcal{L}/2$. Since $\mathcal{L} \succeq 0$, it follows that $Q \succeq \alpha I$.  However, Theorem \ref{thm:reducedacl} guarantees that for each iteration of Algorithm \ref{acl:l1pg}, one has $\text{supp}(\qq_k) \subseteq \mathcal{S}_*$ $\forall k$. Naturally, the function $f$, restricted to vectors with support in $\mathcal{S}_{*}$, has a better strong convexity parameter. Let $\mathcal{L}_{\mathcal{S}_*}$ be the principal sub-matrix of the normalized graph Laplacian $\mathcal{L} = I - D^{-1/2} A D^{-1/2}$ by removing the 
rows and columns with indices in $V\backslash\mathcal{S}_*$. It is clear that such restricted strong convexity parameter, when restricted to all vectors $\qq$ such that $\text{supp}(\qq) \subseteq \mathcal{S}_*$, is $\alpha + (1-\alpha)\lambda_{min}(\mathcal{L}_{\mathcal{S}_*})/2$, which, if $\lambda_{min}(\mathcal{L}_{\mathcal{S}_*}) > 0$, is larger than $\alpha$.
%Note that $\lambda_{min}(\mathcal{L}_{\mathcal{S}_*})$ depends on $\alpha$ via its dependence on ${\mathcal{S}_*}$, which, in turn, is correlated with $\alpha$. 
%However, for any vector $\qq$ such that $\text{supp}(\qq) \subseteq \mathcal{S}_*$ the restricted strong convexity parameter of $f$ can be much larger than $\alpha$. 
%is $\alpha + \frac{1-\alpha}{2}\lambda_{min}(\mathcal{L}_{\mathcal{S}_*})$.

Now consider the local conductance constant, defined in \cite{Chung07_localcutsLAA} as
$$
H (\mathcal{S}):= \min_{S\subset \mathcal{S}} \Phi(S).
$$
Note this latter definition differs from~\eqref{eq:phig} in that $H (\mathcal{S})$ measures the minimum conductance over all subsets of $\mathcal{S}$, as opposed to $\V$.
Suppose $\mathcal{G}$ is connected and let $\|\bss\|_1/\rho \le \mbox{vol}(\G)/2$, which, from Theorem \ref{thm:volume}, implies that $\mbox{vol}(\mathcal{S}_*) \le \mbox{vol}(\G)/2$. 
This is a reasonable assumption since, in the context of local graph clustering, it is not desired for the optimal support, $\mathcal{S}_*$, to have a volume larger than half of that of the whole graph, $\mathcal{G}$.
In \cite{Chung07_localcutsLAA}, a local Cheeger inequality is proved for the Dirichlet eigenvalue $\lambda_{min}(\mathcal{L}_{\mathcal{S}_*})$ of the induced subgraph on $\mathcal{S}_*$. For cases when such induced subgraph is connected, the lower bound given in~\cite{Chung07_localcutsLAA} is in the form of
\begin{equation}\label{eq:localcheeger}
0 < \frac{\left(H (\mathcal{S}_*)\right)^2}{2} \le \lambda_{min}(\mathcal{L}_{\mathcal{S}_*}).
\end{equation}

%The constant $H (\mathcal{S}_*)$ in~\eqref{eq:localcheeger} is positive 
%For this restricted strong convexity parameter to be larger than $\alpha$, we need to have $\lambda_{min}(\mathcal{L}_{\mathcal{S}_*}) > 0$. In fact, if the subgraph of $\G$ induced by the nodes in $\mathcal{S}_*$ is connected, then by using the Perron-Frobenius Theorem for matrix $I - \mathcal{L}_{\mathcal{S}_*}$, it follows that $\lambda_{min}(\mathcal{L}_{\mathcal{S}_*})>0$. It can be shown that, for any tolerance parameter in the termination condition, the latter condition is always satisfied by the output of Algorithm \ref{acl:l1pg}. In other words, the optimal support $\mathcal{S}_*$ of the solution corresponds to a connected induced subgraph of $\G$. Indeed, Step 4 of Algorithm \ref{acl:l1pg} ensures that the procedure only touches the neighbors of the current non-zero nodes. Therefore, if the input reference set of nodes (captured by vector $\bss$) corresponds to connected induced subgraph of $\G$, the support of the output of Algorithm \ref{acl:l1pg} and consequently $\mathcal{S}_*$ correspond to connected induced subgraphs of $\G$.
%if the subgraph of $\G$ induced by the nodes in $\mathcal{S}_*$ is connected \cite{Chung07_localcutsLAA}.
%then by using the Perron-Frobenius Theorem for matrix $I - \mathcal{L}_{\mathcal{S}_*}$, it follows that $\lambda_{min}(\mathcal{L}_{\mathcal{S}_*})>0$. 
Luckily, it can be shown that, for any tolerance parameter in the termination condition, the optimal support $\mathcal{S}_*$ from Algorithm \ref{acl:l1pg} corresponds to a connected induced subgraph of $\G$. Indeed, Step 4 of Algorithm \ref{acl:l1pg} ensures that the procedure only touches the neighbors of the current non-zero nodes. Therefore, if the input reference set of nodes (captured by vector $\bss$) \textcolor{black}{corresponds to connected induced subgraphs} of $\G$, the support of the output of Algorithm \ref{acl:l1pg} and consequently $\mathcal{S}_*$ correspond to connected induced subgraphs of $\G$.
Note that, in the cases where $\G$ is disconnected, the above reasoning still holds as long as $\rho$ is chosen such that $\|\bss\|_1/\rho \le \text{vol}(\tilde{\mathcal{G}})/2$, where $\tilde{\mathcal{G}} \subset \mathcal{G}$ is the largest connected component of $\mathcal{G}$ that \textcolor{black}{includes a reference node, i.e., a node $i$ that satisfies $\bss(i)\neq 0$} (otherwise, for the output of Algorithm \ref{acl:l1pg}, we might have $\mathcal{S}_* = \tilde{\mathcal{G}}$, %the optimal support might be as large as the largest connected component that includes the reference node
which implies $\lambda_{min}(\mathcal{L}_{\mathcal{S}_*})=0$).

Thus, using \eqref{eq:localcheeger}, we can define the restricted strong convexity parameter of $f$ as
\begin{equation}\label{eq:strparam}
\mu := \alpha + \frac{1-\alpha}{4}\left(H (\mathcal{S}_*)\right)^2.
\end{equation}
%which can be larger than $\alpha$ depending on the local conductance constant $H (\mathcal{S}_*)$.
We are not aware of any better lower bound for $\lambda_{min}(\mathcal{L}_{\mathcal{S}_*})$. In fact, we believe that to lower bound this constant, one needs to make some strong assumptions about the target cluster that includes the reference node. As this is not our primary objective in this paper, we leave this for future work.
%However, in case that the assumptions for the combinatorial guarantees of APPR we will discuss later in the section that the same guarantees hold 
%for $\ell_1$-reg. PR.
%If there exists a subset of $\mathcal{S}_*$ that has very small conductance then $H (\mathcal{S}_*)$ is very small as well, otherwise, if there is no better subset than $\mathcal{S}_*$ then $H (\mathcal{S}_*)$ 
%can be as large as $\Phi(\mathcal{S}_*)$. 
%The former case may indicate that the optimal solution of $\ell_1$-reg. PR does not correspond to a cluster of small conductance for the chosen reference node (although there exists one),
%which is due to bad parameter tuning or even due to the spectral properties of $\ell_1$-reg. PR.
%Therefore, in case that $\ell_1$-reg. PR is successful we expect $H (\mathcal{S}_*)$ to be close to $\Phi(\mathcal{S}_*)$.

Using the restricted strong convexity parameter~\eqref{eq:strparam}, Theorem~\ref{thm:aclanalysis} below gives the overall iteration complexity and total running time\footnote{Iteration complexity refers to the worst-case number of iterations to satisfy the termination criterion and running time refers to the total amount of work, i.e., the per-iteration cost times iteration complexity.} of Algorithm~\ref{acl:l1pg}. 

\begin{theorem}
\label{thm:aclanalysis}
%For $\mathcal{S}_*:= \mbox{supp}(\qq_*)$, let $\mathcal{L}_{\mathcal{S}_*}$ be the principal sub-matrix of the normalized graph Laplacian $\mathcal{L} = I - D^{-1/2} A D^{-1/2}$ by removing the 
%rows and columns with indices in $V\backslash\mathcal{S}_*$, and define $$\mu := (\alpha + \frac{1-\alpha}{2}\lambda_{min}(\mathcal{L}_{\mathcal{S}_*})).$$
%{\color{black}Let $\mathcal{S}_*:= \mbox{supp}(\qq_*)$ and define $\mu$ as in \eqref{eq:strparam}.}
Algorithm \ref{acl:l1pg} with $\|\bss\|_\infty \ge \rho$ requires at most 
\begin{equation}
T \in \mathcal{O}\left(\frac{1}{\mu} \log \left( \frac{2}{\epsilon^2 \rho^2 \alpha^2 \min_j d_j} \right)\right),
\label{eq:16}
\end{equation}
iterations 
to converge to a solution that satisfies the termination criterion in Step 2, where $\mu$ is as in~\eqref{eq:strparam}. 
Furthermore, the running time of Algorithm \ref{acl:l1pg} is at most
\textcolor{black}{
\begin{equation}
\mathcal{O}\left(\frac{(|\mathcal{S}_*| + \widehat{\mbox{vol}}(\mathcal{S}_*))}{\mu} \log \left( \frac{2}{\epsilon^2 \rho^2 \alpha^2 \min_j d_j} \right)\right),
\label{eq:running_time}
\end{equation}
where $\mathcal{S}_*$ is defined in~\eqref{eq:opt_support} and $\widehat{\mbox{vol}}(\mathcal{S}_*)$ is the volume of $\mathcal{S}_*$ by assuming that the edges of the graph are unweighted, i.e., the sum of all neighbors for each node in $\mathcal{S}_*$.
If we further suppose that $|\mathcal{S}_*|, \widehat{\mbox{vol}}(\mathcal{S}_*) \in \mathcal{O}(\text{vol}(\mathcal{S}_*))$}, \textcolor{black}{then using Theorem~\ref{thm:volume} and $\|s\|_1=1$~\eqref{eq:running_time} simplifies to
\begin{equation}
\mathcal{O}\left(\frac{2}{\rho\mu} \log \left( \frac{2}{\epsilon^2 \rho^2 \alpha^2 \min_j d_j} \right)\right).
\label{eq:running_time_2}
\end{equation}
}
\end{theorem}
\begin{proof}
Let the assumption about $\bss$ from Theorem \ref{thm:reducedacl} hold. Then from Theorem \ref{thm:reducedacl} we have that $\qq_k\ge 0$ $\forall k$, i.e., we always remain in the the non-negative orthant. Denoting the restriction of $\psi(\qq)$ to $\qq\geq 0$, by 
$$\widehat{\psi}(\qq) := \rho \alpha \ee^TD^{1/2}\qq + f(\qq),$$ 
it follows that $\psi(\qq)=\hat{\psi}(\qq)$ for all $\qq$ in the non-negative orthant. 
From $1$-Lipschitz continuity of $\nabla f$ w.r.t. $\ell_2$ norm, it follows that $\widehat{\psi}$ is also smooth with the same parameter, i.e., $1$. 
Hence, for any $\qq_{k}$ from Algorithm~\ref{acl:l1pg}, we have
\begin{equation}\label{eq:liphatpsi}
\widehat{\psi}(\qq) \leq \psi(\qq_{k}) + (\qq-\qq_{k})^{T} \nabla \widehat{\psi}(\qq_{k}) + \frac{1}{2} \|\qq_{k} - \qq \|_{2}^{2}.
\end{equation}
Since $\qq_{k+1}\ge 0$ (see Theorem \ref{thm:reducedacl}),\ $\qq_{k+1} - \qq_k = I_{S_k} \Delta \qq_k$ and $\Delta \qq_k = - \nabla_{S_k} \widehat{\psi}(\qq_k)$ we have that 
\begin{align}\label{eq:scmin}
{\psi}(\qq_{k+1}) \leq  \psi(\qq_{k}) -\frac{1}{2}\|\nabla_{S_k} \widehat{\psi}(\qq_{k})\|_{2}^{2}.
\end{align}
We have that $f$ is $\mu$-restricted strongly convex when restricted to all vectors $\qq$ such that $\mbox{supp}(\qq) \subseteq \mathcal{S}_*$, where $\mu := (\alpha + (1-\alpha)\lambda_{min}(\mathcal{L}_{\mathcal{S}_*})/2)$.
Therefore, $\psi$ is $\mu$-restricted strongly convex as well and we have 
$$
\psi(\qq_k)-\psi(\qq_*) \le \frac{1}{2\mu}\|g\|_2^2 \quad \forall g\in \partial \psi(\qq_k),
$$
where $\partial \psi(\qq_k)$ is the sub-differential of $\psi$ at $\qq_k$.\ 
Notice that $I_{S_k}\nabla \widehat{\psi}_{S_k}(\qq_{k})$ is a valid sub-gradient of $\psi$ at $\qq_k$. This gives us 
\begin{equation}\label{sclb2}
\psi(\qq_k)-\psi(\qq_*) \le \frac{1}{2\mu}\|\nabla_{S_k} \widehat{\psi}(\qq_{k})\|_{2}^{2}.
\end{equation}
Combining \eqref{eq:scmin} and \eqref{sclb2} and subtracting $\psi(\qq_*)$ from both sides we get
$$
\psi(\qq_{k+1})-\psi(\qq_*) \leq \left(1 - \mu\right)(\psi(\qq_k)-\psi(\qq_*)),
$$
which implies linear convergence.\
Applying the last inequality recursively we get that Algorithm \ref{acl:l1pg} requires at most $T \in \mathcal{O}({(1/\mu)\log ({1}/{\hat{\epsilon}})})$
iterations to obtain a solution $\qq_T$ such that $\psi(\qq_T) - \psi(\qq_*) \le \hat{\epsilon}$.

From \eqref{eq:scmin} we have that
$$
\psi(\qq_*) \leq \psi(\qq_{k}) - \frac{1}{2}\|\nabla_{S_k} \widehat{\psi}(\qq_{k})\|_{2}^{2} \quad \forall k.
$$
Using the above and $\psi(\qq_T) - \psi(\qq_*) \le \hat{\epsilon}$,\ we get $\|\nabla_{S_k} \widehat{\psi}(\qq_{T})\|_\infty^2 \le 2\hat{\epsilon}$, which is equivalent to
$$
-\rho \alpha - \left(\frac{2\hat{\epsilon}}{d_i}\right)^{1/2} \le \frac{\nabla_i f(\qq_T)}{d^{1/2}_i} \le \rho \alpha + \left(\frac{2\hat{\epsilon}}{d_i}\right)^{1/2}
$$
$\forall i\in S_k$. From Theorem \ref{thm:reducedacl} we have that $\nabla_i f(\qq_T) > -\rho \alpha d_i^{1/2}$ $\forall i\in [n] \backslash S_k$.\
Let $\epsilon\in(0,1)$ be the accuracy parameter of Algorithm \ref{acl:l1pg}. As a result, by setting $\hat{\epsilon} := (\epsilon^2 \rho^2 \alpha^2 \min_j d_j)/2$ and using the fact that $\nabla f(\qq_k) \le 0$ $\forall k$ from Lemma \ref{thm:reducedacl}, we get that after 
\begin{equation*}
T \in \mathcal{O}\left(\frac{1}{\mu} \log \left( \frac{2}{\epsilon^2 \rho^2 \alpha^2 \min_j d_j} \right)\right)
\end{equation*}
%By evaluating~\eqref{eq:scmin} at $\qq_T$ and the fact that $\psi(\qq_T) - \psi^* \le \hat{\epsilon}$, we get $\|\nabla \psi(\qq_T)\|_\infty^2 \le 2\hat{\epsilon}$, which is equivalent to
%$$
%-\rho \alpha - \left(\frac{2\hat{\epsilon}}{d_i}\right)^{1/2} \le \frac{\nabla_i f(\qq_T)}{d^{1/2}_i} \le -\rho \alpha + \left(\frac{2\hat{\epsilon}}{d_i}\right)^{1/2}
%$$
%$\forall i$.
%As a result, by setting $\hat{\epsilon} := (\rho^2 \alpha^{2} \min_j d_j)/2$ and using the fact that $\nabla f(\qq_k) \le 0$ $\forall k$ from Theorem \ref{thm:reducedacl}, we get that after 
%\begin{equation}\label{eq:16}
%T \ge \frac{1}{\alpha} \log \left( \frac{2}{\rho^2 \alpha^{2} \min_j d_j} \right)
%\end{equation}
iterations the output of Algorithm \ref{acl:l1pg} satisfies $-(1+\epsilon)\rho \alpha d^{1/2}_i \le \nabla_i f(\qq_T) \le 0$ $\forall i$, which is the termination criterion in Step $2$ of Algorithm \ref{acl:l1pg}. 

\textcolor{black}{From Theorem \ref{thm:reducedacl} we have that $S_k \subseteq \mathcal{S}_*$ and $|S_k| \le |\mathcal{S}_*|$ $\forall k$.
The set $S_k$ in Step $3$ of Algorithm \ref{acl:l1pg} can be updated in $\mathcal{O}(\widehat{\mbox{vol}}(S_{k-1}))$ operations, where $\widehat{\mbox{vol}}(S_{k-1})$ is the volume of $S_{k-1}$ by assuming that the edges of the graph are unweighted, 
 i.e., the sum of all neighbors for each node in $\mathcal{S}_*$.
The quantity $\widehat{\mbox{vol}}(S_{k-1})$ is upper bounded by $\widehat{\mbox{vol}}(\mathcal{S}_*)$. Therefore, Step $3$ costs at most $\mathcal{O}(\widehat{\mbox{vol}}(\mathcal{S}_*))$ operations.}
Step $4$ of Algorithm \ref{acl:l1pg} requires at 
most $\mathcal{O}(|\mathcal{S}_*|)$ operations.\ \textcolor{black}{Similarly, Steps $5$ and $6$ require at most $\mathcal{O}(|\mathcal{S}_*| + \widehat{\mbox{vol}}(\mathcal{S}_*))$ operations}.\  
%This can be upper bounded by $\mathcal{O}(|\mathcal{S}_*|(1+\max_{i\in \mathcal{S}_*} d_i))$.
Finally, Step $7$ does not perform any computations.\ Putting the operations performed in all of the steps together, using the iteration complexity result in \eqref{eq:16} and the result of Theorem \ref{thm:volume}, we get~\eqref{eq:running_time} and~\eqref{eq:running_time_2}.
\end{proof}

\begin{remark} \textcolor{black}{The assumption $|\mathcal{S}_*|, \widehat{\mbox{vol}}(\mathcal{S}_*) \in \mathcal{O}(\text{vol}(\mathcal{S}_*))$} in the latter part of Theorem~\ref{thm:aclanalysis} holds for many types of graphs, e.g., unweighted. Indeed, such assumption is commonly made in the related literature, including APPR in \cite{ACL06} and many others \cite{AL08,ST13,OZ14,KG14, veldticml2016}.\end{remark}

\begin{remark} For unweighted graphs, according to Theorem~\ref{thm:aclanalysis}, the worst-case running time of Algorithm \ref{acl:l1pg} is $\mathcal{O}\left(\log \left( {2}/{(\epsilon^2 \rho^2 \alpha^2)} \right)/(\rho\mu)\right)$ (ignoring small terms and using $\|\bss\|_1 \le 1$), where $\mu$ was defined in~\eqref{eq:strparam}. However,~\cite[Theorems 1 and 5]{ACL06} state that the worst-case running time of APPR is $\mathcal{O}(1/(\rho\alpha))$.
%While we haven't managed to theoretically establish a relationship between $|\mathcal{S}_*| + \mbox{vol}(\mathcal{S}_*)$ and $1/\rho$, in Section \ref{sec:examples}, we provide empirical evidence showing that the number of non-zeros of the output of Algorithm \ref{acl:l1pg} 
%is indeed similar to that of the output of APPR. Since APPR never accesses nodes of the graph that are not
%part of the output solution, we can conjecture that both algorithms, in practice, have roughly the same running time. This conjecture is certainly corroborated by the numerical examples in Section \ref{sec:examples}.
%As a result, our bounds show strict improvement compared to the related ones of APPR, since we get better dependence on $\alpha$ through $\mu$ (recall that $\mu > \alpha$).
Despite the fact that $\mu \geq \alpha$, since~\eqref{eq:running_time_2} involves $H (\mathcal{S}_*)$ as well as a ``$\log$'' factor, it is unfortunately difficult to directly compare the worst-case running time of Algorithm \ref{acl:l1pg} with that of APPR.\end{remark}

It is possible to replace the output of APPR with the solution of~\eqref{l1pg} and still maintain the combinatorial guarantees for PageRank-Nibble as in~\cite[Theorem 7]{ACL06}; see also the discussion in Section \ref{sec:background}.\ 
This can be shown using the fact that ISTA Algorithm \ref{acl:l1pg} for $\ell_1$-regularized PR satisfies the invariance property of APPR (see~\cite[Section 3]{ACL06}).
Moreover, all algorithms at termination satisfy $\|D^{-1/2} \nabla f(q_k)\|_\infty \le \rho \alpha$.
The proof is identical to that of Theorem $7$ in \cite{ACL06} and is, therefore, omitted. 
Relatedly, to ensure that the solutions of Algorithm \ref{acl:l1pg} and APPR share the same theoretical clustering guarantees, the parameter $\rho$ of Algorithm \ref{acl:l1pg} must be set with respect to that of APPR. More specifically, let $\rho, \tilde{\rho}\in(0,1)$ be the parameters of the $\ell_1$-regularized PR problem \eqref{l1pg} and APPR, respectively. Moreover, let the vector $\bss \geq0$ be chosen such that $\bss(i)\ge \max(\rho,\tilde{\rho})$ for all $i$ with $\bss(i)\neq 0$, e.g., $\bss(i)=1$ for the reference node $i$ and zero elsewhere. Then APPR algorithm at termination gives an output which satisfies~\eqref{eq:termcritacl} while Algorithm \ref{acl:l1pg} is terminated when $\|D^{-1/2}\nabla f(\qq_k)\|_\infty \le (1+\epsilon)\rho \alpha$. Hence, one can set $\rho\le\tilde{\rho}/(1+\epsilon)$ to ensure that the termination criterion of Algorithm \ref{acl:l1pg} matches that of APPR; see Section~\ref{sec:examples} for numerical experiments.

\section{Experiments}
\label{sec:examples}
%This section describes experimental results for our implementations of the algorithms described in this paper.

In this section, we numerically demonstrate that $\ell_1$-reg.\ PR problem achieves in practice similar graph cut guarantees as APPR.
The experiments are performed on a single thread of a 64-core machine
with four 2.4 GHz 16-core AMD Opteron 6278 processors.\ The
implementations are written using C++ code and compiled with the g++
compiler version 4.8.0.  We use a set of undirected, unweighted
real-world graphs from the Stanford Network Analysis Project
({\scriptsize \url{http://snap.stanford.edu/data}}), whose sizes are
shown in Table~\ref{table:sizes}.\ 
%All implementations will be part of the Ligra project ({\scriptsize \url{https://github.com/jshun/ligra}}).
\begin{table}[htb]
\caption{Graph inputs. $^\dag$Number of unique undirected edges.}
%\vspace{-3pt}
\centering
\begin{tabular}[!t]{l|c|c}
{Input Graph} & Num. Vertices & Num. Edges$^\dag$ \\
%& Vertices & Edges & & Ecc. \\
\hline\hline
wiki-Talk & 2,394,385 & 4,659,565\\
soc-LJ & 4,847,571 & 42,851,237\\
cit-Patents & 6,009,555 & 16,518,947\\
com-Orkut & 3,072,627 & 117,185,083\\
\end{tabular}
\label{table:sizes}
%\vspace{-6pt}
\end{table}
%\vspace{-0.2cm}
We present the performance of greedy and heuristic versions of APPR and ISTA.
In particular, in the following figures \textsc{APPR greedy} is Algorithm \ref{acl:cd_ppr} where 
in step $3$ we select the $i$'th coordinate with the largest partial derivative $\nabla_i f(q_k)$ in absolute value. \textsc{APPR heuristic} is 
Algorithm \ref{acl:cd_ppr} where we select approximately the $i$'th coordinate with the largest $\nabla_i f(q_k)$ in absolute value.
In particular, a priority queue of coordinates is maintained which initially contains the starting vertex only. 
On each iteration we select the highest-priority coordinate in the queue and update the coordinate and its neighbors accordingly. 
For each neighbor, insert it in the queue if it is above the threshold with priority equal to the chosen coordinate.
Note that this is a heuristic because we select coordinates based on their priority when they are initially inserted in the queue, and do not update their priorities later on.
It is important to mention that the heuristic versions of the algorithms are guaranteed to converge in theory 
but not with linear convergence rate. However, there exist examples where one can maintain the linear convergence rate, as discussed in Section $5$ in \cite{drt06}.\
%The Cheeger-like guarantees on the output quality hold for the heuristic versions as well.
%For block-ISTA we select $20\%$ of $S_k$ $\forall k$ in a greedy fashion.
%For single-ISTA we set $\tau_k=1$ $\forall k$ and we also implement greedy and heuristic versions of it.
%For the greedy versions of block-ISTA we select coordinate based on the absolute values of $\nabla_i f(q_k) + \rho\alpha d_i^{1/2}$,
%as stated in Algorithm \ref{acl:blcd}.

For all experiments we set $s_v = 1$ and zero elsewhere, where the coordinate/node $v$ is chosen 
based on a search of over $10^4$ starting nodes. We used the starting vertex that gave the best conductance.\
We conduct all experiments by fixing $\alpha=0.1$ and choose the $\rho$
values empirically such that we get clusters with at least $100$ nodes each. This agrees with the observations 
in \cite{LLDM11} regarding the size of local clusters in large-scale graphs.

%\subsection{$\ell_1$-reg. PR $\approx$ APPR}
We use the same rounding procedure as the one described in Section 2.2 in \cite{ACL06} for the original APPR algorithm,
which is based on the conductance criterion.\ In Figure \ref{fig:conductanveVsCluster} we present the conductance criterion 
($y$-axis) versus the volume of the clusters ($x$-axis) produced by the sweep procedure in increasing order. All algorithms obtain approximately 
the same conductance value after the rounding procedure. The number of non-zeros of the output for each 
algorithm is given in Table~\ref{table:nnz}. Notice that the output of the $\ell_1$-reg.~PR problem, which is obtained by ISTA,
has at most the same number of non-zeros as the greedy and the heuristic versions of APPR.
\begin{table}[htb]
\caption{Number of non-zeros for the output solution $p_k$ of each algorithm for the 
four experiments in Figure \ref{fig:conductanveVsCluster}.}
%\vspace{-3pt}
\centering
\begin{tabular}[!t]{l|c|c|c}
{Input Graph} &  \textsc{APPR greedy}& \textsc{APPR heur.} & ISTA \\
%& Vertices & Edges & & Ecc. \\
\hline\hline
wiki-Talk      & 326 & 334 & 326 \\
soc-LJ        & 159 & 159 & 159 \\
cit-Patents  & 210 & 211 & 198 \\
com-Orkut & 447 & 448 & 442\\
\end{tabular}
\label{table:nnz}
%\vspace{-6pt}
\end{table}

\begin{figure}[ht]
%\vspace{-10pt}
\centering
\subfigure[{wiki-Talk, $\alpha=0.1$, $\rho=10^{-5}$}]{
\includegraphics[scale=0.45]{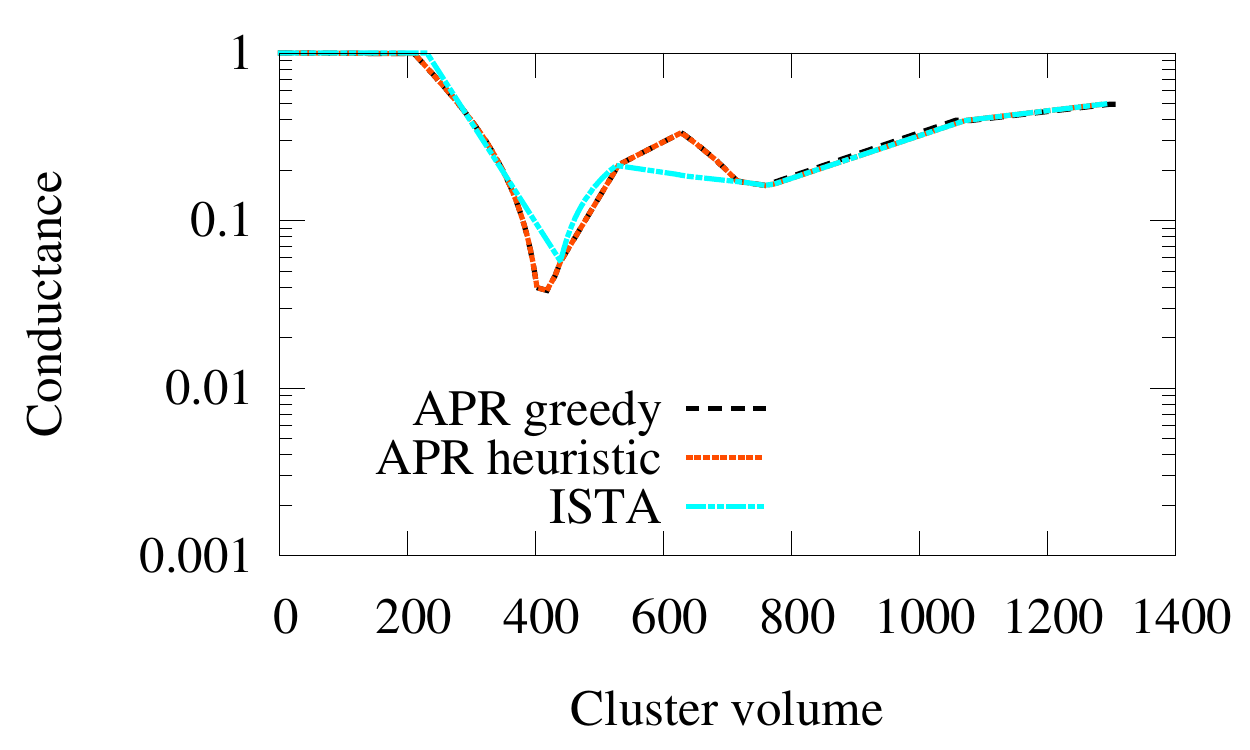}
}
\subfigure[{soc-LJ, $\alpha=0.1$, $\rho=10^{-5}$}]{
\includegraphics[scale=0.45]{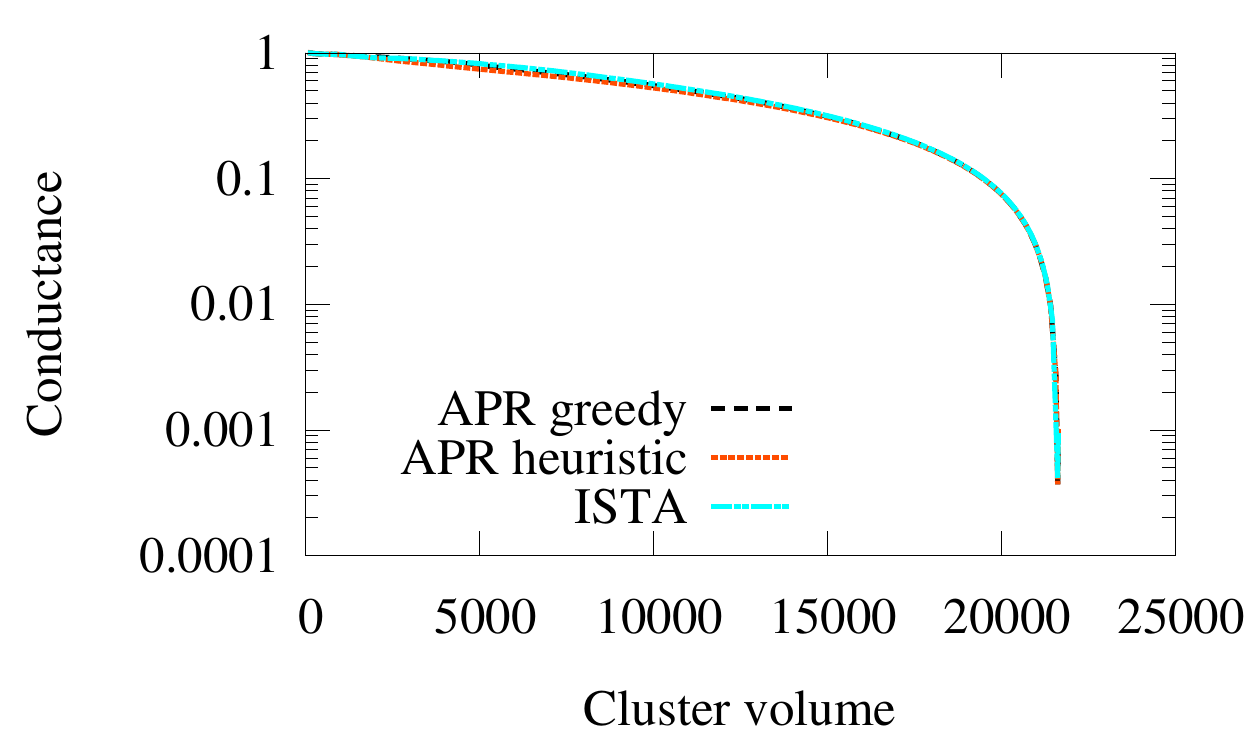}
}
\subfigure[{cit-Patents, $\alpha=0.1$, $\rho=10^{-5}$}]{
\includegraphics[scale=0.45]{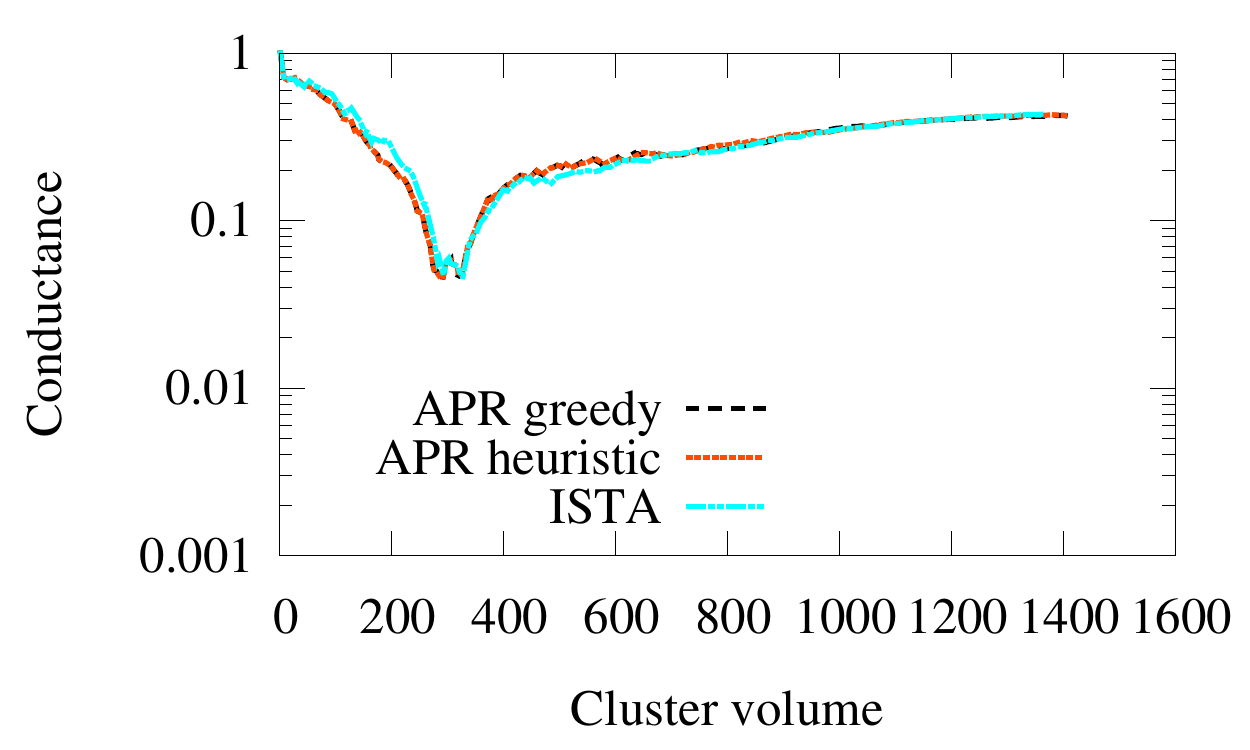}
}
\subfigure[{com-Orkut, $\alpha=0.1$, $\rho=10^{-5}$}]{
\includegraphics[scale=0.45]{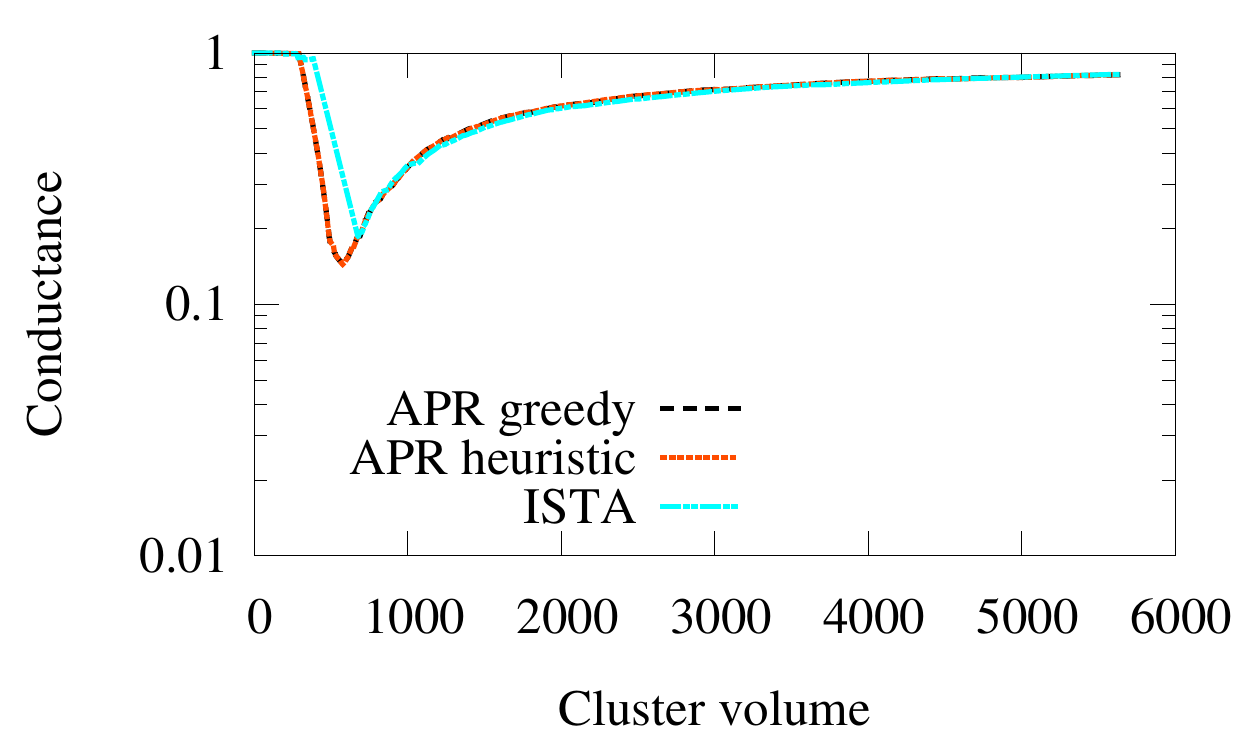}
}
\caption{Conductance vs. cluster volume. The axes of all plots are in log-scale. 
This figure shows the conductance criterion for the clusters which are produced by the sweep procedure 
applied on the output of each algorithm. The volume of the clusters is shown in increasing size.}
\label{fig:conductanveVsCluster}
%\vspace{-2pt}
\end{figure}

%\subsection{Running time}
%{\color{red}
%In Figure \ref{fig:termcritVSruntime} we plot the value of the quantity $\|D^{-1/2} \nabla f(\qq_k)\|_\infty$ against 
%the running time of each algorithm.\ The quantity $\|D^{-1/2} \nabla f(\qq_k)\|_\infty$ is used as a termination 
%criterion for all algorithms.\ 
%Notice that ISTA and APPR have similar running times, this is expected since they also have similar worst-case running time guarantees;
%see discussion after Theorem \ref{thm:aclanalysis}.
%
%\begin{figure}[ht]
%%\vspace{-10pt}
%\centering
%\subfigure[{wiki-Talk, $\alpha=0.1$, $\rho=10^{-5}$}]{
%\includegraphics[scale=0.28]{wiki}
%}
%\subfigure[{soc-LJ, $\alpha=0.1$, $\rho=10^{-5}$}]{
%\includegraphics[scale=0.28]{soc}
%}
%\subfigure[{cit-Patents, $\alpha=0.1$, $\rho=10^{-5}$}]{
%\includegraphics[scale=0.28]{cit}
%}
%\subfigure[{com-Orkut, $\alpha=0.1$, $\rho=10^{-5}$}]{
%\includegraphics[scale=0.28]{orkut}
%}
%\caption{Termination criterion, i.e., $\|D^{-1/2}\nabla f(q_k)\|_\infty$, vs.~running time.\ The axes of all plots are in semi-log-vertical scale.}
%\label{fig:termcritVSruntime}
%%\vspace{-2pt}
%\end{figure}
%}

%\input{figure-runningtime}

%\begin{figure}[!t]
%\includegraphics[width=0.9\columnwidth]{figs/pdfs/objective_cit.pdf}
%\caption{$\|D^{-1/2}\nabla f(q_k)\|_\infty$ versus iteration number on cit-Patents ($\alpha=0.1, \rho=10^{-4}$)}
%\label{fig:objective}
%\end{figure}

%
%\input{figure-varyRho}
%\input{figure-varyAlpha}

\section{Conclusion}
\label{sec:conclusion}
In this paper, we derived and studied a variational formulation of the celebrated local spectral clustering algorithm APPR in~\cite{ACL06}. Through this explicit formulation, we argued that an existing state-of-the-art optimization algorithm, i.e., ISTA \cite{sra2012optimization}, can be applied in a way as to result in a strongly local algorithm, which only requires access to a small portion of the graph. In addition, we showed that the running time of this algorithm only depends on the volume of non-zeros of the solution, as opposed to the entire graph. From a broader perspective, we hope that this variational viewpoint serves as a bridge across two seemingly disjoint fields of graph processing and numerical optimization, and allows one to leverage well-studied, numerically robust, and efficient optimization algorithms for processing today's large graphs. For example, one might be able to apply a modification of accelerated ISTA, i.e. FISTA~\cite{sra2012optimization} to further improve upon the efficiency of local graph clustering algorithms. This can indeed be a direction for future research, which we plan to undertake.

\section*{Acknowledgements}
MM would like to thank the Army Research Office and the Defense Advanced Research Projects Agency for partial support of this work.  
JS was supported by the Miller Institute for Basic Research in Science at UC Berkeley.\ We would also like acknowledge Guy Blelloch at Carnegie Mellon University for providing the machine used for experiments.
%\end{acknowledgements}

% BibTeX users please use one of
\bibliography{references}
\bibliographystyle{spmpsci}      % mathematics and physical sciences
%\bibliographystyle{spphys}       % APS-like style for physics
%\bibliography{}   % name your BibTeX data base

% Non-BibTeX users please use

\end{document}